\documentclass[a4paper,11pt]{amsart}

\usepackage{amsmath,amsthm,amssymb}
\usepackage[mathscr]{eucal}
\usepackage{cite}
\usepackage[bookmarks,bookmarksnumbered]{hyperref}

\numberwithin{equation}{section}

\theoremstyle{plain}
\newtheorem{theorem}{Theorem}[section]

\newtheorem{lemma}[theorem]{Lemma}
\newtheorem{proposition}[theorem]{Proposition}
\newtheorem{corollary}[theorem]{Corollary}
\theoremstyle{definition}
\newtheorem{definition}[theorem]{Definition}

\newtheorem{remark}[theorem]{Remark}

\begin{document}

\title[Unbounded derivations and indecomposability results for II$_1$ factors]{Unbounded derivations, free dilations and indecomposability results for II$_1$ factors}

\author{Yoann Dabrowski and Adrian Ioana}
\thanks {Y.D. was partially supported by ANR grant NEUMANN}
\thanks{A.I. was partially supported by NSF  Grant DMS 1161047}
\address{ Universit\'{e} de Lyon\\
Universit\'{e} Lyon 1\\
Institut Camille Jordan UMR 5208\\
43 blvd. du 11 novembre 1918\\
F-69622 Villeurbanne cedex\\
France}
\email{dabrowski@math.univ-lyon1.fr}
\address{Mathematics Department; University of California, San Diego, CA 90095-1555 (United States).}
\email{aioana@ucsd.edu}

\maketitle

\begin{abstract} We give sufficient conditions, in terms of the existence of unbounded derivations satisfying certain properties, which ensure 
that a II$_1$ factor $M$ is prime or has at most one Cartan subalgebra. For instance, we prove that if there exists a real closable unbounded densely defined derivation $\delta:M\rightarrow L^2(M)\bar{\otimes}L^2(M)$ whose domain contains a non-amenability set, then $M$ is prime. If $\delta$ is moreover ``algebraic'' (i.e. 
its domain $M_0$ is finitely generated,  $\delta(M_0)\subset M_0\otimes M_0$ and $\delta^*(1\otimes 1)\in M_0$), then we show that $M$ has no Cartan subalgebra. We also give several applications to examples from free probability. Finally, we provide a class of countable groups $\Gamma$, defined through the existence of an unbounded cocycle $b:\Gamma\rightarrow \mathbb C(\Gamma/\Lambda)$, for some subgroup $\Lambda<\Gamma$, such that the II$_1$ factor $L^{\infty}(X)\rtimes\Gamma$ has a unique Cartan subalgebra, up to unitary conjugacy, for any free ergodic probability measure preserving (pmp) action $\Gamma\curvearrowright (X,\mu)$.

\end{abstract}

\section {Introduction and statement of main results}

\subsection{Background}
A central theme in the theory of von Neumann algebras is to investigate various decompositions of II$_1$ factors, such as tensor product and Cartan decompositions. Recall that a II$_1$ factor $M$ is {\it prime} if it cannot be written as the tensor product $M=M_1\bar{\otimes}M_2$ of two II$_1$ factors. Also, a maximal abelian von Neumann subalgebra $A\subset M$ is a {\it Cartan subalgebra} if its normalizer, $\mathcal N_{M}(A)=\{u\in\mathcal U(M)|uAu^*=A\}$, generates a weakly dense subalgebra of  $M$.

 The general goal of this paper is to  provide new classes of II$_1$ factors that are prime and have at most one Cartan subalgebra.
We start by giving a short history of results of this type. 

In \cite{Po83}, Popa proved that uncountable free groups give rise to II$_1$ factors that are prime and do not have Cartan subalgebras. The first examples of separable such II$_1$ factors were obtained in the mid 90s as an application of free probability theory.  
Thus, Voiculescu showed that  any II$_1$ factor admitting a generating set whose free  entropy dimension is greater than $1$ has no Cartan subalgebra \cite{Vo95}. In particular, 
the free group factors, $L(\mathbb F_n)$, with $2\leqslant n\leqslant\infty$, do not have Cartan subalgebras.
 Subsequently, Ge proved that the free group factors are also prime \cite{Ge96}.

During the last decade, Popa's deformation/rigidity theory has  generated spectacular progress in the study of II$_1$ factors. In particular, it has led to the first classes of II$_1$ factors that have a unique Cartan subalgebra, up to unitary conjugacy. We highlight here three major advances in this direction and refer the reader to the surveys  \cite{Po07,Va10,Io12b} for more information.    In \cite{Po01} Popa showed that any II$_1$ factor has at most one Cartan subalgebra which satisfies a certain combination of deformation and rigidity properties. Later on, Ozawa and Popa found the first class of II$_1$ factors that have a unique arbitrary Cartan subalgebra \cite{OP07}. More precisely, they showed that  any II$_1$ factor $L^{\infty}(X)\rtimes\mathbb F_n$  arising from a free ergodic profinite pmp action $\mathbb F_n\curvearrowright (X,\mu)$ of a free group $\mathbb F_n$ ($2\leqslant n\leqslant\infty$) has a unique Cartan subalgebra, up to unitary conjugacy. Very recently, Popa and Vaes vastly generalized this result by proving that it holds for any free ergodic  pmp action $\mathbb F_n\curvearrowright (X,\mu)$ \cite{PV11}.

In the last ten years, the primeness and absence of Cartan subalgebras of the free group factors \cite{Vo95,Ge96}  have been generalized and strengthened in many ways.  Firstly, Ozawa proved that II$_1$ factors arising from hyperbolic groups $\Gamma$ are {\it solid}: the relative commutant $A'\cap L(\Gamma)$ of any diffuse subalgebra $A\subset L(\Gamma)$ is amenable \cite{Oz03}. In particular,  $L(\Gamma)$ and all of its non-amenable subfactors are prime. Secondly, using a technique based on closable derivations,
 Peterson was able to show that II$_1$ factors arising from  groups with positive first $\ell^2$-Betti number are prime \cite{Pe06}. 

Using his deformation/rigidity theory, Popa then found a new proof of solidity for $L(\mathbb F_n)$  \cite{Po06b}. Popa's approach relies on the remarkable discovery \cite{Po06a}  that the presence of spectral gap can be viewed as a source of rigidity. This {\it spectral gap rigidity} principle has since been the catalyst behind many developments in deformation/rigidity theory. Thus, it was a crucial ingredient in the finding of II$_1$ factors with a unique Cartan subalgebra \cite{OP07,PV11}. 
In \cite{OP07}, Ozawa and Popa used the spectral gap rigidity principle to show that the free group factors enjoy a  structural property, called {\it strong solidity}, which strengthens both solidity and absence of Cartan subalgebras: the normalizer of any diffuse subalgebra $A\subset L(\mathbb F_n)$ is amenable.
Also using  Popa's spectral gap rigidity principle, Chifan and Sinclair showed that, more generally, the group von Neumann algebra of any icc hyperbolic group is strongly solid \cite{CS11}.

\subsection{Statement of main results} 
  In this paper, we use Popa's deformation/rigidity theory to prove  primeness and absence/uniqueness of Cartan subalgebras for II$_1$ factors in the presence of unbounded  closable derivations satisfying certain regularity properties.
   
 Our first result shows that if a non-amenable II$_1$ factor $M$ admits a closable unbounded derivation into its coarse bimodule which has a ``large"  domain, then $M$ is prime. To make this precise, let us introduce a definition.
If $M$ is a II$_1$ factor, then we say that a finite set $S\subset M$ is a {\it non-amenability set} if 
there exists a  constant $K>0$ such that $\|\xi\|_2\leqslant K\sum_{x\in S}\|x\xi-\xi x\|_2$, for every vector $\xi\in L^2(M)\bar{\otimes}L^2(M)$. 
Note that by Connes' theorem \cite{Co76}, $M$ is non-amenable if and only if it admits a non-amenability set.

\begin{theorem}\label{L^2-rigid} 
Let $M$ be a non-amenable II$_1$ factor and $M_0$ be a weakly dense $*$-subalgebra which contains a non-amenability set for $M$. 
Assume that there exists a real closable unbounded derivation $\delta:M_0\rightarrow (L^2(M)\bar{\otimes}L^2(M))^{\oplus\infty}$. 

Then $M$ is not L$^2$-rigid. In particular, $M$ is prime and does not have property Gamma.   

\end{theorem}
Recall that if $\mathcal H$ is an $M$-$M$ bimodule, then a map $\delta:M_0\rightarrow\mathcal H$ is a {\it derivation} if it verifies that $\delta(xy)=x\delta(y)+\delta(x)y$, for all $x,y\in M_0$.
We say that $\delta$ is {\it bounded} if $\sup_{x\in M_0,\;\|x\|\leqslant 1}\|\delta(x)\|<\infty$. Note that $\delta$ is bounded if and only if it is {\it inner}, i.e. there exists $\xi\in\mathcal H$ such that $\delta(x)=x\xi-\xi x$, for all $x\in M_0$ (see the proof of \cite[Theorem 2.2]{Pe04}).
Also, recall that a II$_1$ factor $M$ has {\it property Gamma} of Murray and von Neumann  \cite{MvN43} if there exists  a sequence $u_n\in\mathcal U(M)$ such that $\tau(u_n)=0$, for all $n$, and $\|u_nx-xu_n\|_2\rightarrow 0$, for all $x\in M$. 

A II$_1$ factor $M$ is {\it L$^2$-rigid} in the sense of Peterson \cite{Pe06} if any semigroup $\phi_t=\exp(-t\delta^*\bar{\delta})$ arising from a real closable densely defined derivation $\delta$ into a multiple of the coarse bimodule converges uniformly to id$_M$ on the unit ball of $M$, as $t\rightarrow 0$. By \cite{Pe06} if a II$_1$ factor is not prime or has property Gamma, then it is $L^2$-rigid. On the other hand, if an icc group $\Gamma$ admits an unbounded cocycle $b:\Gamma\rightarrow \ell^2(\Gamma)^{\oplus\infty}$, then $L(\Gamma)$ is not $L^2$-rigid \cite{Pe06}. Theorem \ref{L^2-rigid} generalizes this fact and provides new examples of non-$L^2$-rigid factors.

Note that if $M_0$ does not contain a non-amenability set for $M$, then Theorem \ref{L^2-rigid} fails in general (see Remark \ref{dense}).

By \cite{Vo95,Ge96}, II$_1$ factors which admit a set of generators whose microstates free entropy is finite, $\chi(X_1,...,X_n)>-\infty$, do not have property Gamma, are prime and do not have Cartan subalgebras. In \cite{Vo98} Voiculescu introduced a non-microstates free entropy $\chi^*(X_1,...,X_n)$. The two entropies satisfy $\chi\leqslant\chi^*$  by \cite{BCG03} and are believed to be equal,  whenever Connes' embedding conjecture holds. Nevertheless, unlike its microstates counterpart, the non-microstates free entropy has not yet found applications to von Neumann algebras.   In particular, it is as open problem whether the above indecomposability results hold under the assumption $\chi^*(X_1,...,X_n)>-\infty$.
In this direction, it was shown in \cite{Da08} that if the assumption that $\chi^*(X_1,...,X_n)>-\infty$ is replaced with the stronger assumption that the free Fisher information is finite $\Phi^*(X_1,...,X_n)<\infty$, then the von Neumann algebra $M$ generated by $\{X_1,...,X_n\}$ is a II$_1$ factor without property Gamma.   

In this paper, we show that if we further strengthen the condition $\Phi^*(X_1,...,X_n)<\infty$ then we can conclude that $M$ is prime and does not have a Cartan subalgebra.
Firstly, as a consequence of Theorem \ref{L^2-rigid} we deduce the following.

\begin{corollary}\label{conjvar}
Let $(M,\tau)$ be a tracial von Neumann algebra which is generated by $n\geqslant 2$ algebraically free self-adjoint elements $X_1,...,X_n$. 
  Assume that  either
  \begin{itemize}
   \item $\mathscr J_p (X_i:\mathbb C\langle X_1,...,X_{i-1},X_{i+1},...,X_n\rangle)$ exists and belongs to $M$, for all $p\in\{1,2\}$ and every $i\in\{1,...,n\}$, or
   \item  $\Phi^*(X_1,...,X_n)<\infty$ and $n\geqslant 3$.
   \end{itemize}

Then $M$ is a non-L$^2$-rigid II$_1$ factor. In particular, $M$ is prime and does not have property Gamma.   
\end{corollary}
To recall the definition of the $p$-th order conjugate variable $\xi_{p,i}=\mathscr J_p (X_i:\mathbb C\langle X_1,...,X_{i-1},X_{i+1},...,X_n\rangle)$, for $p\in\{1,2\}$, let $M_0$ be the $*$-algebra generated by $X_1,...,X_n$.  Let $\delta_i:M_0\rightarrow L^2(M)\bar{\otimes}L^2(M)$ be the partial free difference quotient  derivation given by $\delta_i(X_j)=\delta_{i,j}1\otimes 1$. Then the first and second order conjugate variables are defined as $\xi_{1,i}=\delta_i^*(1\otimes 1)\in L^2(M)$ and $\xi_{2,i}=\delta_i^*(\xi_{1,i}\otimes 1)\in L^2(M)$, whenever these formulas make sense (see \cite[Definition 3.1]{Vo98}).  If the first order conjugate variables $\xi_{1,i}$ exist, then the free Fisher information is given by $\Phi^*(X_1,...,X_n)=\sum_{i=1}^n\|\xi_{1,i}\|_2^2$.

Corollary \ref{conjvar} implies that if $X_1,...,X_n$ are self-adjoint elements of a tracial von Neumann algebra $(M,\tau)$, for some $n\geqslant 2$, and $S_1,...,S_n\in M$ are free semicircular elements which are free from $X_1,...,X_n$,  then  $X_1^{\varepsilon}=X_1+\varepsilon S_1,...,X_n^{\varepsilon}=X_n+\varepsilon S_n$ generate a  prime II$_1$ factor, for any $\varepsilon>0$.  Indeed, by  \cite[Corollary 3.9]{Vo98}  $X_1^{\varepsilon},...,X_n^{\varepsilon}$ satisfy the above assumption on conjugate variables. 
Moreover, by using \cite{Io12a}, we can show that the II$_1$ factor generated by $X_1^{\varepsilon},...,X_n^{\varepsilon}$ does not have a Cartan subalgebra (see Theorem \ref{epsilon}). 
Note that in the case the von Neumann algebra generated by $\{X_1,...,X_n\}$ is embeddable into $R^{\omega}$, the last two facts follow from \cite{Vo95,Vo97,Ge96}.
See Section \ref{reg} for indecomposability results for more general ``regularized" algebras.

Secondly, by assuming a Lipschitz condition on  conjugate variables 
 \cite[Definition 1]{Da10b} we are able to deduce absence of Cartan subalgebras.

\begin{theorem}\label{freedef} Let $(M,\tau)$ be a tracial von Neumann algebra which is generated by $n\geqslant 2$ algebraically free self-adjoint elements $X_1,...,X_n$. Let $M_0$ be the $*$-algebra generated by $X_1,...,X_n$. For $1\leqslant i\leqslant n$,  denote by $\delta_i:M_0\rightarrow L^2(M)\bar{\otimes}L^2(M)$ the free difference quotient $\delta_i(X_j)=\delta_{i,j}1\otimes 1$. 

Let $\delta=(\delta_1,...,\delta_n):M_0\rightarrow (L^2(M)\bar{\otimes}L^2(M))^{\oplus_n}$ and $\bar{\delta}$ be the closure of $\delta$.
Assume that $1\otimes 1$ is in the domain of $\delta_i^*$ and denote $\xi_{i}=\delta_i^*(1\otimes 1)$. Moreover, assume that $\xi_i$ is in the domain of $\bar{\delta}$ and $\bar{\delta}(\xi_i)\in (M\bar{\otimes}M^{op})^{\oplus_n}$, for all $1\leqslant i\leqslant n$. Here, $M^{op}$ denotes the opposite algebra of $M$, and we consider the inclusion $M\bar{\otimes}M^{op}\subset L^2(M\bar{\otimes}M^{op})\cong L^2(M)\bar{\otimes}L^2(M)$.

Then $M$ is a II$_1$ factor which does not have a Cartan subalgebra. Moreover, $M\bar{\otimes}Q$ does not  have a Cartan subalgebra, for any II$_1$ factor $Q$.
\end{theorem}

 Theorem \ref{freedef} generalizes a result of \cite{Da10b} where  by using microstates free entropy techniques and \cite{Vo95,Sh07}  it was shown that if $M$ is embeddable into $R^{\omega }$ then it has no Cartan subalgebras. 
 
In the second part of this paper we establish absence or uniqueness of Cartan subalgebras for II$_1$ factors $M$ admitting certain unbounded ``algebraic" derivations. 

Firstly, we show that, under fairly general conditions, the existence of a finitely generated weakly dense $*$-subalgebra $M_0\subset M$, a von Neumann subalgebra $B\subset M$, and an unbounded derivation $\delta:M_0\rightarrow L^2(\langle M,e_B\rangle)$ such that $\delta(M_0)\subset\text{span}\;(M_0e_BM_0)$ implies that  $M$ has no Cartan subalgebras (see Theorem \ref{nocartan}). Here, $\langle M,e_B\rangle$ denotes Jones' basic construction.

Let us state two  corollaries of this result. 
By Theorem \ref{L^2-rigid} if a II$_1$ factor $M$ admits a closable unbounded derivation $\delta:M_0\rightarrow L^2(M)\bar{\otimes}L^2(M)$ whose domain  contains a non-amenability set, then it is prime. We believe that the existence of such a derivation $\delta$ should also imply that $M$ does not have a Cartan subalgebra. However, proving this seems out of reach with the methods that are currently available.
Nevertheless, as a consequence of Theorem \ref{nocartan} we are able to confirm this conjecture if $\delta$ is algebraic.

\begin{corollary}\label{cartan}
 Let $M$ be a non-amenable II$_1$ factor and $M_0\subset M$ be a finitely generated weakly dense  $*$-subalgebra which contains a non-amenability set for $M$.
Assume that there exists a real unbounded derivation $\delta:M_0\rightarrow L^2(M)\bar{\otimes}L^2(M)$  such that $\delta(M_0)\subset M_0\otimes M_0$,  $1\otimes 1$ is in the domain of $\delta^*$ and  $\delta^*(1\otimes 1)\in M_0$. 
   
Then $M$ does not have a Cartan subalgebra. 
\end{corollary}

Corollary \ref{cartan} generalizes part of \cite[Corollary 1.5]{Io12a}. Indeed, it implies that the free product $M=M_1*M_2$ of any two finitely generated II$_1$ factors does not have a Cartan subalgebra. However,  we are unaware of any example of a II$_1$ factor which satisfies the hypothesis of Corollary \ref{cartan} and is not essentially a free product (see also Remark \ref{stallings}). Note that if $M$ is embeddable into $R^{\omega}$, then Corollary \ref{cartan} also follows from \cite{Sh07}.

As a consequence of Theorem \ref{nocartan} we also provide a general criterion for absence of Cartan subalgebras in amalgamated free product II$_1$ factors. 

\begin{corollary}\label{amalgam}
Let $(M_1,\tau_1)$ and $(M_2,\tau_2)$  be tracial von Neumann algebras with a common von Neumann subalgebra such that ${\tau_1}_{|B}={\tau_2}_{|B}$ and denote $M=M_1*_{B}M_2$. Assume that there exist unitary elements $u\in M_1$ and $v,w\in M_2$ such that $E_B(u)=E_B(v)=E_B(w)=E_B(w^*v)=0$. Suppose that either $uBu^*\perp B$ or $vBv^*\perp B$.

Then $M$ is a II$_1$ factor,  does not have a Cartan subalgebra and does not have property Gamma.
\end{corollary}

For the definition of the amalgamated free product of tracial von Neumann algebras, see \cite{Po93} and \cite{VDN92}.
Following \cite{Po83}, we say that two von Neumann subalgebras $B_1,B_2$ of a tracial von Neumann algebra $(M,\tau)$ are orthogonal  if $\tau(b_1b_2)=\tau(b_1)\tau(b_2)$, for all $b_1\in B_1$ and $b_2\in B_2$.

Finally, we provide a new class of II$_1$ factors that have a unique Cartan subalgebra, up to unitary conjugacy. By \cite[Definition 1.4]{PV11} a countable group $\Gamma$, whose every free ergodic pmp  action $\Gamma\curvearrowright (X,\mu)$ gives rise to a II$_1$ factor  with a unique Cartan subalgebra,  is called $\mathcal C$-{\it rigid} (Cartan rigid). In a recent breakthrough, Popa and Vaes proved that  all weakly amenable groups  with a positive first $\ell^2$-Betti number and  all non-elementary hyperbolic  groups are $\mathcal C$-rigid \cite{PV11,PV12}. Most recently, it was shown in \cite[Theorem 1.1]{Io12a} that any amalgamated free product group $\Gamma=\Gamma_1*_{\Lambda}\Gamma_2$ such that $[\Gamma_1:\Lambda]\geqslant 2$, $[\Gamma_2:\Lambda]\geqslant 3$, and $\cap_{i=1}^mg_i\Lambda g_i^{-1}$ is finite, for some $g_1,...,g_m\in\Gamma$, is $\mathcal C$-rigid.

Our next result gives a cohomological criterion, in terms of the existence of a certain unbounded algebraic cocycle, for a group $\Gamma$ to be $\mathcal C$-rigid.

\begin{theorem}\label{unique}
Let $\Gamma$ be a countable icc group, $\Lambda<\Gamma$ be a subgroup and assume that there exists an unbounded cocycle $b:\Gamma\rightarrow\mathbb C(\Gamma/\Lambda)$ satisfying $b_{|\Lambda}\equiv 0$.
 Additionally, assume that
\begin{itemize}

\item $\cap_{i=1}^mg_i\Lambda g_i^{-1}$ is finite, for some  $g_1,g_2,...,g_m\in\Gamma$,
\vskip 0.02in
\item there exists an increasing sequence $\{\Gamma_n\}_{n\geqslant 1}$ of finitely generated subgroups of $\Gamma$ such that $\cup_{n\geqslant 1}\Gamma_n=\Gamma$ and $b(\Gamma_n)\subset\mathbb C(\Gamma_n\Lambda/\Lambda)$, for all $n\geqslant 1$, 
\vskip 0.02in
\item $L(\Gamma)$  does not have property Gamma and $\Lambda$ is not co-amenable in $\Gamma$. 

\end{itemize}

Then $L^{\infty}(X)$ is the unique Cartan subalgebra of $L^{\infty}(X)\rtimes\Gamma$, up to unitary conjugacy, for any free ergodic pmp action $\Gamma\curvearrowright (X,\mu)$.

\end{theorem}

Since amalgamated free product groups admit such algebraic cocycles, this theorem generalizes \cite[Theorem 1.1]{Io12a} recalled above.  Moreover, it leads to new examples of $\mathcal C$-rigid groups. 

\begin{corollary}\label{HNN}
Let $G$ be a countable group, $\Lambda<G$ be a subgroup, and $\theta:\Lambda\rightarrow G$ be an injective group homomorphism such that $\Lambda\not=G$ and $\theta(\Lambda)\not=G$. Denote by $\Gamma=\text{HNN}(G,\Lambda,\theta)$ the corresponding HNN extension. Assume that $\cap_{i=1}^mg_i\Lambda g_i^{-1}$ is finite, for some $g_1,g_2,...,g_m\in\Gamma$.

Then $L^{\infty}(X)$ is the unique Cartan subalgebra of $L^{\infty}(X)\rtimes\Gamma$, up to unitary conjugacy, for any free ergodic pmp action $\Gamma\curvearrowright (X,\mu)$.

\end{corollary}

Corollary \ref{HNN} strengthens and generalizes the main result of \cite{FV10}. Indeed, \cite[Theorem 1.1]{FV10} shows that if we moreover assume that  $G$ contains a non-amenable subgroup with the relative property (T) or two commuting non-amenable subgroups, and that $\Lambda$ is amenable, then $L^{\infty}(X)\rtimes\Gamma$ has a unique group measure space Cartan subalgebra, for any free ergodic pmp action $\Gamma\curvearrowright (X,\mu)$.

\subsection{Comments on the proofs} Let us say a few words about the proofs of Theorems \ref{L^2-rigid}, Theorem \ref{freedef} and Corollary \ref{cartan}, since they are representative of the proofs of all the results stated above.
Consider a real unbounded closable derivation $\delta:M_0\rightarrow (L^2(M)\bar{\otimes}L^2(M))^{\oplus n}$, for some $1\leqslant n\leqslant\infty$, which satisfies the corresponding regularity conditions. Recall that $M_0$ contains a non-amenability set $S$ for $M$. Our goal is to prove either that $M$ is not $L^2$-rigid  (as in Theorem \ref{L^2-rigid}) or that $M$ has no Cartan subalgebra (as in Theorem \ref{freedef} and Corollary \ref{cartan}).

To this end, we use results from free probability theory to construct a malleable deformation of $M$, in the sense of Popa.
This consists of a tracial von Neumann algebra $\tilde M$ containing $M$ and 
a pointwise $\|.\|_2$-continuous path $\{\alpha_t\}_{t\geqslant 0}$ of $*$-homomorphisms $\alpha_t:M\rightarrow\tilde M$ such that $\alpha_0=\text{id}_M$. Moreover, the pair $(\tilde M,\{\alpha_t\}_{t\geqslant 0})$ is a ``dilation of $\delta$" in the following broad sense: the limit $\lim_{t\rightarrow 0}\frac{1}{t}\|\alpha_t(x)-x\|_2$ exists and is determined by $\delta$, for all $x\in M_0$. 

By \cite{Da10a}, any real closable derivation admits a dilation. Moreover, for derivations into a multiple of the coarse bimodule, such as $\delta$, \cite{Da10a} provides additional information on the dilation. More precisely,  if $M_t$ denotes the von Neumann algebra generated by $M$ and $\alpha_t(M)$, then the $M$-$M$ bimodule $L^2(M_t)\ominus L^2(M)$ is contained in a multiple of the coarse bimodule, for any $t\geqslant 0$. 

If  $M$ is $L^2$-rigid,  then the semigroup $\phi_t=\exp(-t\delta^*\bar{\delta})$ converges uniformly to id$_M$ on the unit ball of $M$, as $t\rightarrow 0$.  This readily entails that $\alpha_t$ converges uniformly to id$_M$, as $t\rightarrow 0$. As a consequence, for every small enough $t$, there exists a unitary $u_t\in\tilde M$ such that $u_t\alpha_t(M)u_t^*\subset M$. 
 Since $S\subset M$ is a non-amenability set, a variation of Popa's spectral gap argument  \cite{Po06a} (see Lemma \ref{uniform}) allows us to find $K>0$ such that $$(*)\;\;\;\;\;\|\alpha_t(x)-E_M(\alpha_t(x))\|_2\leqslant K\sum_{y\in S}\|\alpha_t(y)-y\|_2,\;\;\text{for all}\;\;\; x\in (M)_1\;\;\text{and}\;\;t\geqslant 0.$$ 

Since $S\subset M_0$ and $\alpha_t$ dilates $\delta$, this inequality implies that $\delta$ is bounded. This is a contradiction, proving that $M$ is not $L^2$-rigid, as claimed by Theorem \ref{L^2-rigid}.

Now, assume that $\delta$ is the free difference quotient and the conjugate variables satisfy a Lipschitz condition (as in Theorem \ref{freedef}), or $\delta$ is algebraic (as in Corollary \ref{cartan}). Then results from \cite{Da10b} and \cite{Sh07} imply that the dilation algebra $\tilde M$ can be taken equal to $M*L(\mathbb F_{\infty})$. In this case, we say that $\delta$ admits a ``free dilation". 

We then use techniques from Popa's deformation/rigidity theory to prove that $M$ does not have a Cartan subalgebra. In particular, we employ the recent work  \cite{Io12a} (which notably uses \cite{PV11}) on the structure of normalizers of subalgebras of amalgamated free product algebras. By  combining  \cite{Io12a} with  Lemma \ref{uniform} we show that if $M$ has a Cartan subalgebra, then $(*)$ holds. As above, this provides a contradiction.
 
\subsection{Organization of the paper} Besides the introduction, this paper has seven other sections. In Section 2 we record several notions and results that we will later use. In Section 3 we recall known results on dilating derivations. 
Sections 4-8 are devoted to the proofs of our main results.

\subsection{Acknowledgments} We are grateful to Cyril Houdayer, Jesse Peterson and Stefaan Vaes for helpful discussions and useful comments.

\section{Preliminaries}
\subsection{Terminology} 
We work with {\it  tracial} von Neumann algebras $(M,\tau)$, i.e. von Neumann algebras $M$  endowed with a faithful, normal, tracial state $\tau$. We denote by $\|x\|_2=\tau(x^*x)^{1/2}$  the  2-norm associated to $\tau$ and by $\|x\|$ the operator norm.
 We denote by $\mathcal Z(M)$ the {\it center} of $M$, by $\mathcal U(M)$ the {\it group of unitaries} of $M$ and by $(M)_1=\{x\in M|\hskip 0.02in\|x\|\leqslant 1\}$ the {\it unit ball} of $M$. We always assume that $M$ is {\it separable}, unless it is a subalgebra of an ultraproduct algebra.

A tracial von Neumann algebra $(M,\tau)$ is called {\it amenable} if there exists a net $\xi_n\in L^2(M)\bar{\otimes}L^2(M)$ such that $\langle x\xi_n,\xi_n\rangle\rightarrow \tau(x)$ and $\|x\xi_n-\xi_n x\|_2\rightarrow 0$, for every $x\in M$. By Connes' celebrated theorem \cite{Co76}, $M$ is amenable if and only if it is approximately finite dimensional.

For a free ultrafilter $\omega$ on $\mathbb N$, the {\it ultraproduct} algebra $M^{\omega}$ is defined as the quotient  $\ell^{\infty}(\mathbb N,M)/\mathcal I$, where $\mathcal I\subset\ell^{\infty}(\mathbb N,M)$ is the closed ideal of $x=(x_n)_n$ such that $\lim_{n\rightarrow\omega}\|x_n\|_2=0$. As it turns out, $M^{\omega}$ is a tracial von Neumann algebra, with its canonical trace given by $\tau_{\omega}((x_n)_n)=\lim_{n\rightarrow\omega}\tau(x_n)$.

If $M$ and $N$ are tracial von Neumann algebras, then an $M$-$N$ {\it bimodule} is a Hilbert space $\mathcal H$ endowed with commuting normal $*$-homomorphisms $\pi:M\rightarrow \mathbb B(\mathcal H)$ and $\rho:N^{op}\rightarrow \mathbb B(\mathcal H)$. For $x\in M,y\in N$ and $\xi\in\mathcal H$ we denote $x\xi y=\pi(x)\rho(y)(\xi)$. If $M,N,P$ are tracial von Neumann algebras, $\mathcal H$ and $\mathcal K$ be $M$-$N$ and $N$-$P$ bimodules, respectively, then $\mathcal H{\otimes}_N\mathcal K$ denotes the {\it Connes tensor product}  endowed with the natural $M$-$P$ bimodule structure (see \cite{Po86}).

Let $Q\subset M$ be a von Neumann subalgebra. {\it Jones' basic construction} $\langle M,e_Q\rangle$ is defined as the von Neumann subalgebra of $\mathbb B(L^2(M))$ generated by $M$ and the orthogonal projection $e_Q$ from $L^2(M)$ onto $L^2(Q)$.
Recall that $\langle M,e_Q\rangle$ has a faithful semi-finite  trace given by $Tr(xe_Qy)=\tau(xy)$ for all $x,y\in M$. We denote by $L^2(\langle M,e_Q\rangle)$ the associated Hilbert space and endow it with the natural $M$-bimodule structure. Note that $L^2(\langle M,e_Q\rangle)\cong L^2(M){\otimes}_QL^2(M)$, as $M$-$M$ bimodules.

Finally, if $S$ is a subset of a von Neumann algebra $\mathcal M$, then a state $\phi$ on $\mathcal M$ is said to be $S$-{\it central} if it satisfies $\phi(xT)=\phi(Tx)$, for all $x\in S$ and $T\in\mathcal M$.

\subsection {Intertwining-by-bimodules} We next recall from  \cite [Theorem 2.1 and Corollary 2.3]{Po03}  Popa's powerful {\it intertwining-by-bimodules} technique (see also \cite[Appendix C]{Va06}).

\begin {theorem}\cite{Po03}\label{corner} Let $(M,\tau)$ be a separable tracial von Neumann algebra and $P,Q\subset M$ be two (not necessarily unital) von Neumann subalgebras. 
Then the following are equivalent:

\begin{itemize}

\item There exist  non-zero projections $p\in P, q\in Q$, a $*$-homomorphism $\phi:pPp\rightarrow qQq$  and a non-zero partial isometry $v\in qMp$ such that $\phi(x)v=vx$, for all $x\in pPp$.

\item There is no sequence $u_n\in\mathcal U(P)$ satisfying $\|E_Q(xu_ny)\|_2\rightarrow 0$, for all $x,y\in M$.
\end{itemize}

If one of these conditions holds true,  then we say that {\it a corner of $P$ embeds into $Q$ inside $M$} and write $P\prec_{M}Q$.
\end{theorem}
If $M$ is not separable, then this statement holds true after we replace the sequence $u_n$ with a net.

\subsection {Relative amenability}\label{relativeamen}

\begin{definition} \cite[Definition 2.2]{OP07} Let $(M,\tau)$ be a tracial von Neumann algebra and let  $P\subset pMp$, $Q\subset M$ be  von Neumann subalgebras.
We say that
 $P$ is {\it amenable relative to $Q$ inside $M$} if there exists a net $\xi_n\in L^2(p\langle M,e_Q\rangle p)$ such that $\langle x\xi_n,\xi_n\rangle\rightarrow \tau(x)$, for every $x\in pMp$, and $\|y\xi_n-\xi_n y\|_2\rightarrow 0$, for every $y\in P$. 
 By \cite[Theorem 2.1]{OP07}, this condition is equivalent to the existence of a $P$-central state $\phi$ on $p\langle M,e_Q\rangle p$ such that $\phi_{|pMp}=\tau_{|pMp}$. 
 \end{definition}

\begin{remark}
Let $\Lambda<\Gamma$  be countable subgroups. By \cite[Proposition 3.5]{AD95}, $L(\Gamma)$ is amenable relative to $L(\Lambda)$ if and only if $\Lambda$ is {\it co-amenable} in $\Gamma$: there is a $\Gamma$-invariant state on $\ell^{\infty}(\Gamma/\Lambda)$.
\end{remark}

The failure of an algebra to be amenable (or amenable relative to some other algebra) can therefore be viewed as a source of ``spectral gap rigidity". The notion of spectral gap rigidity has been introduced by Popa and has been used to great effect for instance in \cite{Po06a,Po06b,OP07}. Motivated by this, we introduce the following:

\begin{definition}
 Let $(M,\tau)$ be a tracial von Neumann algebra and $P,Q\subset M$ be  von Neumann subalgebras.

\begin{enumerate}
\item A finite set $S\subset M$ is called a {\it non-amenability set for} $M$ if there exists a constant $K>0$ such that $\|\xi\|_2\leqslant K\sum_{y\in S}\|y\xi-\xi y\|_2$, for every $\xi\in L^2(M)\bar{\otimes}L^2(M)$. 
\item A finite set $S\subset P$ is 
called a {\it non-amenability set for $P$ relative to $Q$ inside $M$} if there exists a constant $K>0$ such that  $\|\xi\|_2\leqslant K\sum_{y\in S}\|y\xi-\xi y\|_2$,  for every $\xi\in L^2(\langle M,e_Q\rangle)$.
\end{enumerate}
\end{definition}

\begin{remark}
 If $M$ has a non-amenability set, then $M$ has no amenable direct summand. If $M$ is a II$_1$ factor, then by Connes' theorem \cite{Co76} the converse is  true:  $M$ has a non-amenability set if and only it is non-amenable.
Note also that if there is a non-amenability set for $P$ relative to $Q$, then $Pp$ is not amenable relative to $Q$, for any projection $p\in P'\cap M$.
\end{remark}

We will later need the following result  which is an easy consequence of \cite[Section 2.2]{OP07} (see also \cite[Section 2.5]{PV11}).

\begin{lemma}\label{gap}\cite{OP07}
Let $(M,\tau)$ be a tracial von Neumann algebra and let $P,Q\subset M$ be  von Neumann subalgebras. Let $\mathcal H$ be a $Q$-$P$ bimodule.

\begin{enumerate}
\item Assume that $Pp$ is not amenable relative to $Q$, for any non-zero projection $p\in P'\cap M$. Then for any net of vectors $\xi_n\in L^2(M)\otimes_Q\mathcal H$ satisfying $\sup_n\|x\xi_n\|\leqslant \|x\|_2$, for all $x\in M$, and $\|y\xi_n-\xi_n y\|\rightarrow 0$, for all $y\in P$, we have that $\|\xi_n\|\rightarrow 0$. 

\item If $S\subset P$ is a non-amenability set for $P$ relative to $Q$ inside $M$, then there exists a constant $\kappa>0$ such that $\|\xi\|\leqslant \kappa\sum_{y\in S}\|y\xi-\xi y\|$, for all $\xi\in L^2(M){\otimes}_Q\mathcal H$. 
 \end{enumerate}
\end{lemma}

\begin{proof} The first assertion is a rephrasing of \cite[Lemma 2.3]{Io12a}.

To prove the second assertion, let $S$ be a non-amenability set for $P$ relative to $Q$.
Assuming that the conclusion fails, we can find a sequence of unit vectors $\xi_n\in L^2(M){\otimes}_Q\mathcal H$ such that $\|y\xi_n-\xi_n y\|\rightarrow 0$, for all $y\in S$. Choose a state on $\ell^{\infty}(\mathbb N)$, denoted $\lim_n$,  extending the usual limit. Also, consider the normal $*$-homomorphism $\pi:
\langle M,e_Q\rangle\rightarrow\mathbb B(L^2(M)\otimes_Q\mathcal H)$ given by $\pi(T)(\xi\otimes_Q\eta)=T(\xi)\otimes_Q\eta$.

Define  $\psi:\langle M,e_Q\rangle \rightarrow\mathbb C$ by letting $\psi(T)=\lim_{n}\langle\pi(T)\xi_n,\xi_n\rangle$. Then $\psi$ is a $S$-central state. Moreover, $\psi$ is $C^*(S)$-central, where $C^*(S)$ denotes the $C^*$-algebra generated by $S$. A standard procedure (see the proof of \cite[Theorem 2.1]{OP07}) implies the existence of a net of unit vectors $\eta_i\in  L^2(\langle M,e_Q\rangle)$ such that $\|y\eta_i-\eta_i y\|\rightarrow 0$, for all $y\in\mathcal U(C^*(S))$. Thus, $\|y\eta_i-\eta_i y\|\rightarrow 0$, for all $y\in S$, contradicting the non-amenability of $S$.
\end{proof}

The next lemma is a variant of Popa's spectral gap argument \cite{Po06a}. It will be later used (e.g. in the proof of Theorem \ref{L^2-rigid}) to deduce boundedness of a derivation $\delta$ from the uniform convergence of the semigroup $\phi_t=\exp(-t\delta^*\bar{\delta})$.

\begin{lemma}\label{uniform} Let $(\tilde M,\tau)$ be a tracial von Neumann algebra and $M$ be a von Neumann subalgebra. Let $P,Q\subset M$  be von Neumann subalgebras. Assume that the $M$-$M$ bimodule $L^2(\tilde M)\ominus L^2(M)$ is isomorphic to $L^2(M)\otimes_{Q}\mathcal K,$ for some $Q$-$M$ bimodule $\mathcal K$.

 Let $\alpha_n:M\rightarrow\tilde M$, $n\geqslant 1$, be trace preserving $*$-homomorphisms such that $\|\alpha_n(x)-x\|_2\rightarrow 0$, for all $x\in M$.  Assume that $p_n\in\alpha_n(P)'\cap\tilde M$ is a projection and $v_n\in\tilde M$ is a unitary such that $\alpha_n(P)p_n\subset v_nM{v_n}^*$, for all $n\geqslant 1$.
\begin{enumerate}
\item If $Pp$ is not amenable relative to $Q$ inside $M$, for every non-zero projection $p\in P'\cap M$, then $\sup_{x\in (P)_1}\|(\alpha_n(x)-E_M(\alpha_n(x)))p_n\|_2\rightarrow 0$, as $n\rightarrow\infty$.

\item If $S\subset P$ is a non-amenability set for $P$ relative to $Q$ inside $M$, then  there exists a constant $C>0$ such that for all $n\geqslant 1$ we have $$\|(\alpha_n(x)-E_M(\alpha_n(x)))p_n\|_2\leqslant C\sum_{y\in S}\|(\alpha_n(y)-y)p_n\|_2,\;\;\text{for all}\;\; x\in (P)_1.$$
\end{enumerate}

\end{lemma}
\begin{proof}
For $x\in P$ and $n\geqslant 1$, define $\beta_n(x)=v_{n}^*\alpha_n(x)p_nv_n\in M$. Denote $\mathcal H=L^2(\tilde M)\ominus L^2(M)$. Let $\mathcal H_n$ be the $M$-$P$ bimodule which is equal to $\mathcal H$ endowed with the bimodule structure given by $y\cdot\xi\cdot x=y\xi\beta_n(x)$.
Then  $\mathcal H_n$ is isomorphic to $L^2(M)\otimes_{Q}\mathcal K_n$, where $\mathcal K_n$ is equal to $\mathcal K$ endowed with the $Q$-$P$ bimodule structure given by $y\cdot\xi\cdot x=y\xi\beta_n(x)$.

Define the $Q$-$P$ bimodule $\tilde{\mathcal K}=\oplus_{n\geqslant 1}\mathcal K_n$ and let us treat separately the two assertions.

(1) In this case, Lemma \ref{gap} (1) implies that any net $\xi_n\in L^2(M)\otimes_Q\tilde{\mathcal K}$ satisfying $\|x\cdot\xi_n\|\leqslant \|x\|_2$ for every $x\in M$, and $\|x\cdot \xi_n-\xi_n\cdot x\|\rightarrow 0$, for all $x\in P$, must verify $\|\xi_n\|\rightarrow 0$.
Define $\xi_n=p_nv_n-E_M(p_nv_n)$, then $\|x\xi_n\|_2\leqslant\|x\|_2$, for all $x\in M$. 
If  we view $\xi_n$ as an element of $\mathcal H_n$, then for every $x\in P$ we have that $$\|x\cdot\xi_n-\xi_n\cdot x\|=\|x\xi_n-\xi_n\beta_n(x)\|_2\leqslant \|xp_nv_n-p_nv_n\beta_n(x)\|_2=$$ $$\|(x-\alpha_n(x))p_nv_n\|_2\rightarrow 0.$$

Since $\mathcal H_n$ is isomorphic to a $M$-$P$ sub-bimodule of $L^2(M)\otimes_Q\tilde{\mathcal K}$, we conclude that $\|\xi_n\|_2\rightarrow 0$. It follows that for all $x\in (P)_1$ and $n\geqslant 1$ we have  $$\|\alpha_n(x)p_n-E_M(\alpha_n(x))p_n\|_2=\|\alpha_n(x)p_nv_n-E_M(\alpha_n(x))p_nv_n\|_2\leqslant $$ $$2\|\xi_n\|_2+\|(1-E_M)(\alpha_n(x)p_nv_n)\|_2=2\|\xi_n\|_2+\|(1-E_M)(p_nv_n\beta_n(x))\|_2\leqslant $$ $$2\|\xi_n\|_2+\|(1-E_M)(p_nv_n)\|_2= 3\|\xi_n\|_2.$$

Since $\|\xi_n\|_2\rightarrow 0$ and $x\in (P)_1$ is arbitrary, this proves the first assertion.

(2) Assume that
 $S$ is a non-amenability set for $P$ relative to $B$.  Lemma \ref{gap} (2) implies that
we can find  $\kappa>0$ such that any vector $\xi\in L^2(M)\otimes_B\tilde{\mathcal K}$ verifies $\|\xi\|\leqslant \kappa\sum_{y\in S}\|y\cdot\xi-\xi\cdot y\|$. Thus, for all $n\geqslant 1$ and $\xi\in L^2(\tilde M)\ominus L^2(M)$, we have that 
$ \|\xi\|_2\leqslant \kappa\sum_{y\in S}\|y\xi-\xi\beta_n(y)\|_2.$

Denote $\delta_n=\sum_{y\in S}\|(\alpha_n(y)-y)p_n\|_2$. Then we have  $\sum_{y\in S}\|yp_nv_n-p_nv_n\beta_n(y)\|_2=\delta_n$ and thus $$\sum_{y\in S}\|y(p_nv_n-E_M(p_nv_n))-(p_nv_n-E_M(p_nv_n))\beta_n(y)\|_2\leqslant\delta_n.$$
By combining the last two inequalities we conclude that $\|p_nv_n-E_M(p_nv_n)\|_2\leqslant \kappa\delta_n$, for all $n\geqslant 1$.
Together with the estimate from the proof of part (1), we get that 
if $x\in (P)_1$ and $n\geqslant 1$, then  $\|(\alpha_n-E_M(\alpha_n(x)))p_n\|_2\leqslant 3\|p_nv_n-E_M(p_nv_n)\|_2\leqslant 3\kappa\delta_n.$
Thus, the second assertion holds for $C=3\kappa$.
\end{proof}

\subsection{Property Gamma} A II$_1$ factor $M$ has {\it property Gamma}  of Murray and von Neumann \cite{MvN43} if there exists a sequence of unitaries $u_n\in M$ with $\tau(u_n)=0$ such that $\|u_nx-xu_n\|_2\rightarrow 0$, for all $x\in M$. If $\omega$ is a free ultrafilter on $\mathbb N$, then property Gamma is equivalent to $M'\cap M^{\omega}=\mathbb C1$. 
By a well-known result of Connes \cite[Theorem 2.1]{Co76} property Gamma is also equivalent to 
 the existence of a net of unit vectors $\xi_n\in L^2(M)\ominus\mathbb C1$ satisfying $\|x\xi_n-\xi_n x\|_2\rightarrow 0$, for all $x\in M$.

Therefore, the failure of property Gamma implies the existence of a non-Gamma set in the sense of the following definition that was introduced in  \cite[Definiton 3.1]{Pe04} and was also motivated by \cite[Remark 4.1.6]{Po86}.

\begin{definition}\cite{Pe04} Let $M$ be a II$_1$ factor. A finite set $S\subset M$ is called a {\it non-Gamma set for $M$} if there exists $K>0$ such that $\|\xi\|_2\leqslant \sum_{y\in S}K\|y\xi-\xi y\|_2$, for all $\xi\in L^2(M)\ominus\mathbb C1$.

\end{definition}
\begin{remark}
By \cite[Theorem 2.1]{Co76} any II$_1$ factor $M$ without property Gamma has a non-Gamma set. Note, however, that it is not always possible to find a non-Gamma set for $M$ inside a given weakly dense $*$-subalgebra of $M$. Recall that a countable group $\Gamma$ is {\it inner amenable} if the unitary representation of $\Gamma$ on $\ell^2(\Gamma\setminus\{e\})$ given by conjugation has almost invariant vectors. Vaes recently found an example of an icc group $\Gamma$ which is inner amenable (hence $\mathbb C\Gamma$ does not contain a non-Gamma set for $L(\Gamma$)) such that $L(\Gamma)$ does not have property Gamma \cite{Va09}. \end{remark}

 The next result follows easily from \cite{Co76} but for the reader's convenience we include a proof.

\begin{lemma}\cite{Co76}\label{non-amen}
Let $M$ be a II$_1$ factor and $S\subset M$ be a finite set closed under adjoint. 

If $S$ is a non-Gamma set for $M$, then $S$ is non-amenability set for $M$.
\end{lemma}

\begin{proof}   Let $S$ be a non-Gamma set for $M$. Assume by contradiction that $S$ is not a non-amenability set. Thus we can find a sequence of unit vectors $\xi_n\in L^2(M)\bar{\otimes} L^2(M)$ such that $\|x\xi_n-\xi_n x\|_2\rightarrow 0$, for all $x\in S$. Choose a state on $\ell^{\infty}(\mathbb N)$, denoted $\lim_n$,  extending the usual limit. Define $\psi:\mathbb B(L^2(M))\rightarrow\mathbb C$ by letting $\psi(T)=\lim_n\langle (T\otimes 1)\xi_n,\xi_n\rangle$. 

Then $\psi$ is an $S$-central state, hence $\phi=\psi_{|M}:M\rightarrow\mathbb C$ is an  $S$-central state.  Moreover, $\phi$ is central under the $C^*$-algebra $C^*(S)$ generated by $S$.  Let $\eta_i\in L^1(M)$ be a net of positive norm one elements such that 
$\tau(x\eta_i)\rightarrow\phi(x)$, for all $x\in M$. Since $\phi$ is $C^*(S)$-central, for all $u\in \mathcal U(C^*(S))$ we have that $\tau(x(u\eta_i u^*-\eta_i))\rightarrow 0$, for all $x\in M$. Thus, $u\eta_i u^*-\eta_i\rightarrow 0$, in the weak topology, for all $u\in \mathcal U(C^*(S))$. The Hahn-Banach theorem implies that, after passing to convex combinations, we may assume that we have $\|u\eta_iu^*-\eta_i\|_1\rightarrow 0$ in addition to $\tau(x\eta_i)\rightarrow\phi(x)$, for every $x\in M$.

The Powers-St$\o$rmer inequality (see \cite[Proposition 6.2.4]{BO08}) gives that $\|u\eta_i^{1/2}u^*-\eta_i^{1/2}\|_2\rightarrow 0$, for all $u\in \mathcal U(C^*(S))$. Hence $\|y\eta_i^{1/2}-\eta_i^{1/2}y\|_2\rightarrow 0$, for all $y\in S$.
Since $S$ is a non-Gamma set, we derive that $\|\eta_i^{1/2}-c_i\cdot 1\|_2\rightarrow 0$, where $c_i=\langle\eta_i^{1/2},1\rangle$. Applying Powers-St$\o$rmer again yields that $\|\eta_i-c_i^2\cdot 1\|_1\rightarrow 0$. This implies that $c_i^2\tau(x)\rightarrow\phi(x)$, for all $x\in M$, hence $\phi=\tau$. 
Since $\psi$ is $C^*(S)$-central and $\psi_{|M}=\tau$, we get that $\psi$ is central under the von Neumann algebra $W^*(S)$  generated by $S$. Thus $W^*(S)$ is amenable, contradicting the fact that it is a II$_1$ factor without property Gamma.
\end{proof} 

\subsection{Mixing bimodules} Next, we recall the notion of mixing bimodules introduced in \cite[Definition 2.3]{PS09}.
 
 \begin{definition}\label{mixing}\cite{PS09} Let $(M,\tau)$ be a tracial von Neumann algebra. We say that an $M$-$M$ bimodule $\mathcal H$ is {\it mixing} if 
for any sequence $a_n\in (M)_1$ such that $a_n\rightarrow 0$, weakly, we have  $$\sup_{x\in (M)_1}|\langle a_n\xi x,\eta\rangle|\rightarrow 0,\\\;\;\text{and}\;\;\sup_{x\in (M)_1}|\langle x\xi a_n,\eta\rangle|\rightarrow 0,\;\;\text{as $n\rightarrow\infty$},\;\;\text{for all}\;\;\; \xi,\eta\in\mathcal H.$$
\end{definition}
 
 The coarse $M$-$M$ bimodule $L^2(M)\bar{\otimes} L^2(M)$ is clearly mixing.
 Also, let $\Gamma$ be a countable group and $\pi:\Gamma\rightarrow\mathcal U(\mathcal K)$ be a mixing unitary representation. Then it is easy to see that $\mathcal H=\mathcal K\otimes\ell^2(\Gamma)$ is a mixing $L(\Gamma)$-$L(\Gamma)$ bimodule (with its natural bimodule structure, defined as in \cite[1.1.4.]{Po01}).

\subsection {Normalizers of subalgebras of amalgamated free product von Neumann algebras}
We will also need the following variant of \cite[Theorem 1.6]{Io12a} which is a hybrid between Theorems 1.6 and 5.2 from \cite{Io12a}.

\begin{theorem}\label{general}\cite{Io12a} Let $(M_1,\tau_1)$ and $(M_2,\tau_2)$ be two tracial von Neumann algebras with a common von Neumann subalgebra $B$ such that ${\tau_1}_{|B}={\tau_2}_{|B}$ and denote $M=M_1*_{B}M_2$. Let $(Q,\tau)$ be a tracial von Neumann algebra and $A\subset M\bar{\otimes}Q$ be an amenable von Neumann subalgebra. Denote by $P=\mathcal N_{M\bar{\otimes}Q}(A)''$ the von Neumann algebra generated by the normalizer of $A$ in $M\bar{\otimes}Q$. 

Assume that there are a group $\mathcal U$ and homomorphisms $\rho_1:\mathcal U\rightarrow\mathcal U(M),\rho_2:\mathcal U\rightarrow\mathcal U(Q)$ such that
\begin{itemize}
\item $\rho_1(u)\otimes \rho_2(u)\in P$, for all $u\in\mathcal U$, and 
\item the von Neumann subalgebra $P_0\subset M$ generated by $\rho_1(\mathcal U)$ satisfies $P_0'\cap M^{\omega}=\mathbb C1$.
\end{itemize}
Then one of the following conditions holds true:

\begin{enumerate}
\item $A\prec_{M\bar{\otimes}Q}B\bar{\otimes}Q$.
\item $P_0\prec_{M}M_i$, for some $i\in\{1,2\}$.
\item$P_0$ is amenable relative to $B$ inside $M$.
\end{enumerate}
\end{theorem} 

\begin{proof} For completeness, let us briefly indicate how the result follows from \cite{Io12a}.

Define $\mathcal M=M\bar{\otimes}Q$, $\mathcal M_1=M_1\bar{\otimes}Q$,  $\mathcal M_2=M_2\bar{\otimes}Q$ and $\mathcal B=B\bar{\otimes}Q$. Then $\mathcal M=\mathcal M_1*_{\mathcal B}\mathcal M_2$. 
Further, we define $\tilde{\mathcal M}={\mathcal M}*_{\mathcal  B}(\mathcal B\bar{\otimes}L(\mathbb F_2))$
 and let $\{\theta_t\}_{t\in\mathbb R}\subset\text{Aut}(\tilde{\mathcal M})$ be the free malleable deformation \cite{IPP05} (see e.g. \cite[Section 2.5]{Io12a}).  Let $\{u_g\}_{g\in\mathbb F_2}$ denote the canonical unitaries and define $N\subset\tilde{\mathcal M}$ to be the von Neumann subalgebra generated by $\cup_{g\in\mathbb F_2}u_g\mathcal Mu_g^*$. Then $N$ is normalized by $\{u_g\}_{g\in\mathbb F_2}$ and $\tilde{\mathcal M}=N\rtimes\mathbb F_2$.

Now,  notice that if $t\in (0,1)$ then $\theta_t(P)\subset\mathcal N_{\tilde{\mathcal M}}(\theta_t(A))$.   S. Popa and S. Vaes' dichotomy \cite[Theorem 1.6]{PV11}  implies that either $\theta_t(A)\prec_{\tilde{\mathcal M}} N$ or $\theta_t(P)$ is amenable relative to $N$. Thus, we conclude that  we are in one of the following two cases:

{\bf Case 1}. $\theta_t(A)\prec_{\tilde{\mathcal M}} N$, for some $t\in (0,1)$.

{\bf Case 2}. $\theta_t(P)$ is amenable relative to $N$, for all $t\in (0,1)$.

In the first case, \cite[Theorem 3.2]{Io12a} implies that either $A\prec_{\mathcal M}\mathcal B$ or $P\prec_{\mathcal M}\mathcal M_i$, for some $i\in\{1,2\}$.
If the first alternative holds, then (1) is true. If $P\prec_{\mathcal M}\mathcal M_i$, then $P_0\prec_{M}M_i$ and hence (2) is true. Indeed, if $P_0\nprec_{M}M_i$, then  by the proof of \cite[Corollary 2.3]{Po03} we can find a sequence of unitaries $u_n\in\mathcal U$ such that $\|E_{M_i}(a\rho_1(u_n)b)\|_2\rightarrow 0$, for all $a,b\in M$. But then it is clear that $\|E_{\mathcal M_i}(a(\rho_1(u_n)\otimes\rho_2(u_n))b)\|_2\rightarrow 0$, for all $a,b\in \mathcal M$. This contradicts the assumption that $P\prec_{\mathcal M}\mathcal M_i$.

In the second case, \cite[Theorem 5.2]{Io12a} directly implies that either (2) or (3) hold.
\end{proof}

\section{Derivations and free dilations} 
In this section we record several results about derivations and their dilations.

Let $(M,\tau)$ be a tracial von Neumann algebra, $M_0\subset M$  a weakly dense $*$-subalgebra, and $\mathcal H$ a $M$-$M$ bimodule. A map $\delta:M_0\rightarrow \mathcal H$ is a {\it derivation} if  $\delta(xy)=x\delta(y)+\delta(x)y$, for all $x,y\in M_0$. We assume  that $\delta$ is {\it closable}  as an unbounded operator $\delta:L^2(M)\rightarrow \mathcal H$. We also suppose that $\delta$ is {\it real}, i.e. there exists a conjugate-linear isometric involution $\mathcal J$ on $\mathcal H$ satisfying $\mathcal J(x\delta(y)z)=z^*\delta(y^*)x^*$, for all $x,y,z\in M_0$. When $\mathcal H=L^2(\mathcal M)$, for some semi-finite von Neumann algebra $\mathcal M$ containing $M$, we assume that $\mathcal J$ is given by $\mathcal J(x)=x^*$. In this case, $\delta$ is real if and only if $\delta(x^*)=\delta(x)^*$, for all $x\in M$.

Now, denote by $\bar{\delta}$ the closure of $\delta$ and by $D(\bar{\delta})\subset L^2(M)$ its domain.
By \cite{Sa89} and \cite{DL92}, $D(\bar{\delta})\cap M$ is a $*$-subalgebra and $\bar{\delta}_{|D(\bar{\delta})\cap M}$ is a derivation. Further,  
$\Delta=\delta^*\bar{\delta}$  gives rise to a semigroup of completely positive maps on $M$. More precisely, $\phi_t=\exp(-t\Delta):M\rightarrow M$ are unital, trace preserving, completely positive maps satisfying $\phi_t\circ\phi_s=\phi_{t+s}$, for all $t,s>0$, and $\|\phi_t(x)-x\|_2\rightarrow 0$, as $t\rightarrow 0$, for every $x\in M$. Additionally, since $\delta$ is real, we have that $\phi_t$ is symmetric for every $t>0$: $\tau(\phi_t(x)y)=\tau(x\phi_t(y))$, for all $x,y\in M$.

 Recently it was proved that the semigroup $\{\phi_t\}_{t>0}$ admits a dilation in a larger tracial von Neumann algebra $\tilde M\supset M$ (see \cite[Theorem 24]{Da10a}). Here we state this result in the case  when $\mathcal H$ is a multiple of the coarse $M$-$M$ bimodule.  In this case, \cite[Proposition 26]{Da10a} provides additional information on certain $M$-$M$ sub-bimodules of $L^2(\tilde M)$.

\begin{theorem}\cite{Da10a}\label{dilation}
Let $(M,\tau)$ be a tracial von Neumann algebra and $M_0\subset M$ be a weakly dense $*$-subalgebra. Let $\delta:M_0\rightarrow (L^2(M)\bar{\otimes}L^2(M))^{{\oplus\infty}}$ be a real closable derivation. 
Let $\Delta=\delta^*\bar{\delta}$ and consider the semigroup of completely positive maps $\phi_t=\exp(-t\Delta):M\rightarrow M$.

Then there exists a tracial von Neumann algebra $\tilde M$ which contains $M$ and $*$-homomorphisms $\alpha_t:M\rightarrow\tilde M$ such that $\phi_t=E_M\circ\alpha_t$, for all $t>0$. 
Moreover,  denote by  $M_t\subset\tilde M$ the von Neumann subalgebra generated by $M$ and $\alpha_t(M)$. Then  the $M$-$M$ bimodule $L^2(M_t)\ominus L^2(M)$ is isomorphic to a sub-bimodule of  $(L^2(M)\bar{\otimes} L^2(M))^{\oplus\infty}$, for every $t>0$.
\end{theorem}
 
In the sequel we will also need the following technical result. 
 
\begin{lemma}\label{estimate}
Consider the notations from Theorem \ref{dilation}. Then for every $x\in D(\bar{\delta})$ we have that $$\frac{1}{t}\|\alpha_t(x)-\phi_t(x)\|_2^2\rightarrow 2\|\delta(x)\|_2^2\;\;\;\text{and}\;\;\;\frac{1}{t}\|\alpha_t(x)-x\|_2^2\rightarrow 2\|\delta(x)\|_2^2,\;\;\;\text{as}\;\;\; t\rightarrow 0.$$
\end{lemma}

\begin{proof} Let $t>0$ and recall $\phi_t=\exp(-t\Delta)$, where $\Delta=\delta^*\bar{\delta}$.  
 By  combining the identity $\text{id}-\phi_t=\int_0^t\Delta\circ\phi_s\; \text{d}s$ with the fact that $x\in D(\bar{\delta})=D(\Delta^{1/2})$ we get that $$\langle x-\phi_t(x),x\rangle=\int_0^t\langle\Delta(\phi_s(x)),x\rangle\;\text{d}s=\int_0^t\langle \phi_s(\Delta^{1/2}(x)),\Delta^{1/2}(x)\rangle\;\text{d}s.$$
 
 Since $\phi_s(\Delta^{1/2}(x))\rightarrow \Delta^{1/2}(x)$, in $\|.\|_2$, as $s\rightarrow 0$, we conclude that \begin{equation}\label{conv}\frac{1}{t}\langle x-\phi_t(x),x\rangle\rightarrow \|\Delta^{1/2}(x)\|_2^2=\|\delta(x)\|_2^2,\;\;\text{as}\;\;t\rightarrow 0.\end{equation}
 
 Finally, since $\phi_t(x)=E_M(\alpha_t(x))$, we get that $\|\alpha_t(x)-\phi_t(x)\|_2^2=\|x\|_2^2-\|\phi_t(x)\|_2^2=\langle x-\phi_{2t}(x),x\rangle$. Also, we have that $\|\alpha_t(x)-x\|_2^2=2\langle x-\phi_t(x),x\rangle$. Together with equation \ref{conv} these identities yield the conclusion.
\end{proof}

In the next section, the dilations from  Theorem \ref{dilation} will be used to prove that certain II$_1$ factors $M$ are prime. On the other hand, in order to deduce that $M$ does not have Cartan subalgebras,  we will additionally need to  know that the dilation ``lives" in the free product $\tilde M=M*L(\mathbb F_{\infty})$. In the rest of this section, we recall two results in this direction.

Shlyakhtenko showed that  any ``algebraic" derivation $\delta:M_0\rightarrow L^2(M)\bar{\otimes}L^2(M)$  gives rise, via exponentiation, to a one-parameter group of automorphisms of $M*L(\mathbb Z)$ \cite[Proposition 2]{Sh07}. 
Here we note the following straightforward generalization of this result.

\begin{proposition}\label{dimagen}
Let  $(M,\tau)$ be a tracial von Neumann algebra, $B\subset M$ be a von Neumann subalgebra and $M_0\subset M$ be a weakly dense $*$-subalgebra.  Assume that $\delta:D(\delta)\rightarrow L^2(\langle M,e_B\rangle)$ is a real derivation whose domain, denoted $D(\delta)$, contains both $B$ and $M_0$ such that
\begin{itemize}
\item $\delta(M_0)\subset\;\text{span}(M_0e_BM_0)$.
\item $e_B$ is in the domain of $\delta^*$ and  $\delta^*(e_B)\in M_0$.
\item  $\delta(b)=0$, for all $b\in B$.

\end{itemize}

Assume that $M_0$ is finitely generated. More generally, assume that  $M_0=\cup_{n\geqslant 1}M_n$, where $M_n$ are finitely generated $*$-algebras satisfying  $M_n\subset M_{n+1}$ and $\delta(M_n)\subset\;\text{span}(M_ne_BM_n)$, for all $n\geqslant 1$.

Denote $\tilde M=M*_B(B\bar{\otimes}L(\mathbb Z))$ and let $s\in L(\mathbb Z)$ be a generating $(0,1)$ semicircular element. Also, let $L^2(\langle M,e_B\rangle)\ni\xi\rightarrow\xi\#s\in \overline{\text{span}(MsM)}^{\|.\|_2}\subset L^2(\tilde M)$ be the unique isomorphism of $M$-$M$ bimodules  sending $e_B$ to $s$.

Then there exists a one-parameter group of automorphisms $\{\alpha_t\}_{t\in\mathbb R}$  of  $\tilde M$ such that 

$$\|\frac{1}{t}(\alpha_t(x)-x)-\delta(x)\#s\|_2\rightarrow 0,\;\;\;\text{as}\;\; t \rightarrow 0,\;\;\;\text{for all}\;\; x\in M_0.$$

\end{proposition}

The proof  is an easy adaptation of the proof of \cite[Proposition 2]{Sh07} and can be derived by combining results from \cite[Section 3]{Sh00}.
Nevertheless, for the reader's convenience, we will sketch a proof.

\begin{proof} 

If $b\in B$, then $\delta(b)=0$. This implies that $\delta^*(e_Bb)=\delta^*(e_B)b$ and $\delta^*(be_B)=b\delta^*(e_B)$. Since $e_Bb=be_B$, we deduce that $[\delta^*(e_B),b]=0$.

Let $D(\tilde\delta)$  be the weakly dense $*$-subalgebra of $\tilde M$ generated by $M_0\cup B\cup\{s\}$. Since $\delta(x)\#s\in D(\tilde\delta)$, for all $x\in M_0\cup B$, and $\delta^*(e_B)\in M_0$, we can define  $\tilde\delta:M_0\cup B\cup\{s\}\rightarrow D(\tilde\delta)$ by letting 

$$\tilde\delta(x)=\delta(x)\#s,\;\;\text{for all}\;\; x\in M_0\cup B,\;\;\text{and}\;\;\tilde\delta(s)=-\delta^*(e_B).$$

Since $[\delta^*(e_B),B]=0$ and $\delta_{|B}\equiv 0$,  it is easy to see that $\tilde\delta$ extends to a derivation $\tilde\delta:D(\tilde\delta)\rightarrow D(\tilde\delta)$.
Also, since $\delta(x^*)=\delta(x)^*$, for all $x\in M_0$, and $\delta^*(e_B)$ is self-adjoint, we deduce that $\tilde\delta(x^*)=\tilde\delta(x)^*$, for all $x\in D(\tilde\delta)$. Moreover, we have that:
\vskip 0.05in
{\bf Claim 1.} $\tau(\tilde\delta(x))=0$, for all $x\in D(\tilde\delta)$.
\vskip 0.05in

{\it Proof of Claim 1.} Denote by $\mathcal M$ the $*$-subalgebra of $M$ generated by $M_0$ and $B$. 
For $n\geqslant 1$, define $s_n=s^n-\tau(s^n)$. Then $D(\tilde\delta)$ is the linear span of $$\mathcal M\cup\{x_1s_{n_1}x_2...x_ks_{n_k}x_{k+1}|x_1,x_{k+1}\in \mathcal M,x_2,...,x_{k}\in \mathcal M\ominus B,n_1,...,n_k\geqslant 1\}.$$ 
Since $\delta(M_0)\subset\text{span}(M_0e_BM_0)$, we get that $\tilde\delta(\mathcal M)\subset\text{span}(\mathcal Ms\mathcal M)$. Hence, $\tau(\tilde\delta(x))=0$, for all $x\in\mathcal M$.
 Thus, in order to prove the claim, it suffices to show that $\tau(\tilde\delta(x))=0$, for every $x$ of the form $x=x_1s_{n_2}x_2...x_ks_{n_k}x_{k+1}$, for some $k\geqslant 1$, $x_1,x_{k+1}\in\mathcal M,x_2,...,x_{k}\in\mathcal M\ominus B$ and $n_1,...,n_k\geqslant 1$. 
 
 Below, we sketch the proof of this fact in the case when $k$ is even, leaving the (similar) case when $k$ is odd to the reader. Assume therefore that
  $k=2l$, for some $l\geqslant 1$.
 
 Notice first that by freeness it follows that for all $i\in\{1,...,k+1\}\setminus\{l\}$, $y\in ML(\mathbb Z)M$ and every $j\in\{1,...,k\}\setminus\{l,l+1\}$, $z\in L(\mathbb Z)ML(\mathbb Z)$, we have that   \begin{equation}\label{free1}\tau(x_1s_{n_1}...s_{n_{i}}ys_{n_{i+1}}...s_{n_k}x_{{k+1}})=0\;\;\;\text{and}\;\; \tau(x_1s_{n_1}...x_{n_j}zx_{n_{j+1}}...s_{n_k}x_{{k+1}})=0.\end{equation}

If $n\geqslant 1$,  then $\tilde\delta(s_n)=\tilde\delta(s^n)=\sum_{i=0}^{n-1}s^i\tilde\delta(s)s^{n-1-i}.$
Thus, $\tilde\delta(s_n)\in\;$span$(L(\mathbb Z)ML(\mathbb Z))$. Also, recall that $\tilde\delta(x_0)\in\;$span$(ML(\mathbb Z)M)$, for all $x_0\in \mathcal M$. 
By combining these facts with equation \ref{free1}, and using Leibniz's rule for $\tilde\delta$, it follows that \begin{equation}\label{trace}\tau(\tilde\delta(x))=\tau(x_1s_{n_1}...x_l\tilde\delta(s_{n_l}x_{l+1}s_{n_{l+1}})x_{l+2}...s_{n_k}x_{k+1}).\end{equation}

Next, we denote by $\mathcal K\subset\tilde M\ominus B$ the set of alternating words in $M\ominus B$ and $L(\mathbb Z)\ominus\mathbb C1$, which start or begin with an element from $L(\mathbb Z)\ominus\mathbb C1$. Again, by freeness it is easy to see that \begin{equation}\label{free2}\tau(x_1s_{n_1}...x_lyx_{l+2}...s_{n_k}x_{n_{k+1}})=0,\;\;\text{for all}\;\;y\in\mathcal K.\end{equation}

Now, if $b\in B$, then $\tau(\tilde\delta(s)b)=\tau(\delta^*(e_B)b)=-Tr(e_B\delta(b))=0$ and thus $\tilde\delta(s)\in M\ominus B$.
In combination with the formula for $\tilde\delta(s_n)$, we derive that 
$\tilde\delta(s_{n_l}),\tilde\delta(s_{n_{l+1}})\in\;$span$(L(\mathbb Z)(M\ominus B)L(\mathbb Z))$. Also, since $x_{l+1}\in\mathcal M$, we have that $\tilde\delta(x_{l+1})\in\;$span$(M(L(\mathbb Z)\ominus\mathbb C1)M)$. Using these relations and Leibniz's rule it follows that $\tilde\delta(s_{n_l}x_{l+1}s_{n_{l+1}})-\tau(\tilde\delta(s_{n_l}x_{l+1}s_{n_{l+1}}))$ belongs to the linear span of $\mathcal K$.

Combining this fact with \ref{trace} and \ref{free2} yields that $\tau(\tilde\delta(x))=\tau(\tilde\delta(s_{n_l}x_{l+1}s_{n_{l+1}}))\tau(x_1s_{n_1}...x_lx_{l+2}...s_{n_k}x_{k+1})$ and reduces Claim 1 to proving the following:

\vskip 0.05in
{\bf Claim 2.} 
$\tau(\tilde\delta (s_mcs_n))=0$, for all $m,n\geqslant 1$ and $c\in \mathcal M\ominus B$.
\vskip 0.05in

{\it Proof of Claim 2.}
First, since $\tilde\delta(s_n)=\sum_{i=0}^{n-1}s^i\tilde\delta(s)s^{n-1-i}$, for all $n\geqslant 1$, we get that \begin{equation}\label{1}\tau(\tilde\delta(s_mcs_n))=\sum_{i=0}^{m-1}\tau(s^i\tilde\delta(s)s^{m-1-i}cs_n)+\tau(s_m\tilde\delta(c)s_n)+\sum_{j=0}^{n-1}\tau(s_mcs^j\tilde\delta(s)s^{n-1-j}).\end{equation}

Note that $c,\tilde\delta(s)\in M\ominus B$ and $\tau(\tilde\delta(s)c)=-\tau(\delta^*(e_B)c)$. Since $\tilde\delta(c)\in\text{span}(MsM)$, for every $x\in L(\mathbb Z)$ we  have that \begin{equation}\label{2}\tau(\tilde\delta(c)x)=\tau(xs)\tau(\tilde\delta(c)s)=\tau(xs)\langle \delta(c)\#s,e_B\#s\rangle=\tau(xs)\langle\delta(c),e_B\rangle=\end{equation}
$$ \tau(xs)\tau(c\delta^*(e_B))=-\tau(\tilde\delta(s)c)\tau(xs).$$
Altogether, by combining equations \ref{1} and \ref{2} we get that \begin{equation}\label{free3}\tau(\tilde\delta(s_mcs_n))=\tau(\tilde\delta(s)c)[\sum_{i=0}^{m-1}\tau(s_ns^i)\tau(s^{m-1-i})-\tau(s_ns_ms)+\sum_{j=0}^{n-1}\tau(s_ms^j)\tau(s^{n-1-j})].\end{equation}

Now, let $\partial:\mathbb C\langle s\rangle\rightarrow \mathbb C\langle s\rangle\otimes\mathbb C\langle s\rangle$ be the difference quotient derivation given by 
$\partial(s)=1\otimes 1$. Since $s$ is  $(0,1)$ semicircular, we have that  $\partial^*(1\otimes 1)=s$ (see \cite[Proposition 3.8]{Vo98}) and hence $(\tau\otimes\tau)(\partial(s^p))=\tau(s^{p+1})$, for all $p$.  
Thus, we get that $$\sum_{i=0}^{m-1}\tau(s_ns^i)\tau(s^{m-1-i})+\sum_{j=0}^{n-1}\tau(s_ms^j)\tau(s^{n-1-j})=$$ $$(\tau\otimes\tau)[s^n\partial(s^m)-\tau(s^n)\partial(s^m)+\partial(s^n)s^m-\tau(s^m)\partial(s^n)]=$$ $$\tau(s^{m+n+1})-\tau(s^n)\tau(s^{m+1})-\tau(s^m)\tau(s^{n+1}).$$

Since the last term is equal to $\tau(s_nss_m)$ by equation \ref{free3} we conclude that $\tau(\tilde\delta(s_mcs_n))=0$. This finishes the proof of Claim 2 and hence of Claim 1.
\hfill$\square$

Now, let $\tilde M_0\subset\tilde M$ be the $*$-subalgebra generated by $M_0\cup\{s\}$. Then $\tilde M_0\subset D(\tilde\delta)$ and $\tilde\delta(\tilde M_0)\subset\tilde M_0$. 

If $M_0$ is finitely generated, then $\tilde M_0$ is finitely generated.
By using the fact that  $\tilde\delta(x^*)=\tilde\delta(x)^*$, for all $x\in\tilde M_0$, and Claim 1, \cite[Proposition 3.3 and Corollary 3.7]{Vo01} imply that $\tilde\delta$ exponentiates to a one-parameter group $\alpha_t=e^{t\tilde\delta}$ of trace preserving automorphisms of $\tilde M$. Moreover, $\alpha_t$ satisfies the convergence required in the conclusion.

Finally, assume that $M_0$ is the increasing union of $*$-algebras $M_n$ satisfying $\delta(M_n)\subset\; $span$(M_ne_BM_n)$. Let $n_0\geqslant 1$ such that $\delta^*(e_B)\in M_{n_0}$. 
For every $n$, let $\tilde M_n$ be the $*$-algebra generated by $M_n\cup\{s\}$. Then for every $n\geqslant n_0$, we have that $\tilde\delta(\tilde M_n)\subset\tilde M_n$. By the above, $\alpha_t=e^{t\tilde\delta}$ defines a one-parameter group of trace preserving automorphisms of the weak closure of $\tilde M_n$. Since the increasing union $\cup_{n\geqslant 1}\tilde M_n$ is weakly dense in $\tilde M$, the conclusion follows.
\end{proof}

\begin{remark}\label{madef}
Proposition \ref{dimagen} can be used to recover several known constructions of malleable deformations, in the sense of Popa. Let us give two such examples (see also \cite[Example 1]{Sh07}).

 (1) Let $M=M_1*_{B}M_2$ be an amalgamated free product of tracial von Neumann algebras. Denote by $D(\delta)$ the $*$-algebra generated by $M_1$ and $M_2$. Define a derivation  $\delta:D(\delta)\rightarrow L^2(\langle M,e_B\rangle)$  by letting  $\delta(x)=i[x,e_B]$, if $x\in M_1$, and $\delta(x)=0$, if $x\in M_2$. Then $\delta$ is real and satisfies the assumptions of Proposition \ref{dimagen}. For $t\in\mathbb R$, let $u_t=\exp(its)$. The resulting one-parameter group of automorphisms of $\tilde M=M*_{B}(B\bar{\otimes}L(\mathbb Z))$ is given by $\alpha_t(x)=u_txu_t^*$, if $x\in M_1$, and $\alpha_t(x)=x$, if $x\in M_2\cup L(\mathbb Z)$. This is a variant of the {\it free malleable deformation} of $M$ introduced in \cite{IPP05}.

(2) Let $Q\subset P$ be tracial von Neumann algebras and $\theta:Q\rightarrow P$ be a $*$-homomorphism. From this data, an HNN extension $M=\text{HNN}(P,Q,\theta)$ was constructed in  \cite[Section 3]{FV10}. Briefly,  $M$ is a tracial von Neumann algebra generated by $P$ and a unitary element $u$ such that $uxu^*=\theta(x)$, for all $x\in Q$. Denote by $D(\delta)$ the $*$-algebra generated by $P$ and $u$. Then it is easy to see that  $\delta:D(\delta)\rightarrow L^2(\langle M,e_Q\rangle)$ given by $\delta(x)=0$, if $x\in P$, and $\delta(u)=iue_Q$, defines a real derivation  which satisfies the assumptions of Proposition 3.3.  For $t\in\mathbb R$, let $v_t=\exp(its)$. Then the one-parameter group of automorphisms of $\tilde M=M*_Q(Q\bar{\otimes}L(\mathbb Z))$ provided by Proposition \ref{dimagen} satisfies $\alpha_t(x)=x$, if $x\in P\cup L(\mathbb Z)$, and $\alpha_t(u)=uv_t$.  This recovers the malleable deformation of $M$ introduced in \cite[Section 3.5]{FV10}.
\end{remark}

We are grateful to Jesse Peterson for pointing out to us the following remark.

\begin{remark}\label{stallings}
In the case of group algebras, the existence of unbounded algebraic derivations implies strong restrictions on the structure of the group. Let $\Gamma$ be an infinite, finitely generated countable  group and assume that there exists an unbounded derivation $\delta:\mathbb C\Gamma\rightarrow\mathbb C\Gamma\otimes\mathbb C\Gamma$. Define $b:\Gamma\rightarrow\mathbb C\Gamma\otimes\mathbb C\Gamma$ by letting $b(g)=\delta(u_g)u_g^*$. Then we have that $b(gh)=b(g)+u_gb(h)u_g^*$, for all $g,h\in\Gamma$. Now,  the representation of $\Gamma$ on $\mathbb C\Gamma\otimes\mathbb C\Gamma$ by conjugation is isomorphic to the left regular representation $\lambda$ of $\Gamma$ on $\oplus_{n=1}^{\infty}\mathbb C\Gamma$. Thus, we obtain a cocycle $c=(c_n):\Gamma\rightarrow\oplus_{n=1}^{\infty}\mathbb C\Gamma$.

Since $\delta$ is unbounded (hence not inner), it is easy to see that 
 not all of the cocycles $c_n:\Gamma\rightarrow\mathbb C\Gamma$ can be inner. By \cite[Lemma 2]{BV97} this yields that $\Gamma$ has at least two ends. Stallings' theorem now implies that $\Gamma$ is either an amalgamated free product or an HNN extension over a finite subgroup. Thus, if $\Gamma$ is moreover torsion free, then  it is the free product $\Gamma=\Gamma_1*\Gamma_2$ of two infinite groups. 
\end{remark}

We end this section with a result from \cite[Corollary 25]{Da10b} which shows that under a Lipschitz conjugate variables condition, the von Neumann algebra $M$ generated by $n$ self-adjoint elements $X_1,...,X_n$, admits a deformation into $M*L(\mathbb F_{\infty})$.

\begin{theorem}\cite{Da10b}\label{free} Let $(M,\tau)$ be a tracial von Neumann algebra generated by $n\geqslant 2$ self-adjoint elements $X_1,...,X_n$. Let $M_0$ be the $*$-algebra generated by $X_1,...,X_n$. For every $1\leqslant i\leqslant n$, let $\delta_i:M_0\rightarrow L^2(M)\bar{\otimes}L^2(M)$ be the partial free difference quotient  $\delta_i(X_j)=\delta_{i,j}X_i$. Denote $\delta=(\delta_1,...,\delta_n):M_0\rightarrow (L^2(M)\bar{\otimes}L^2(M))^{\oplus_n}$ and let $\bar{\delta}$ be the closure of $\delta$.

Assume that $1\otimes 1$ is in the domain of $\delta_i^*$ and denote by $\xi_i=\delta_i^*(1\otimes 1)$ the corresponding conjugate variable. Moreover, assume that $\xi_i$ is in the domain of $\bar{\delta}$ and $\bar{\delta}(\xi_i)\in (M\bar{\otimes}M^{op})^{\oplus_n}$, for all $i\in\{1,...,n\}$. Here, $M^{op}$ denotes the opposite algebra of $M$, and we consider the inclusion $M\bar{\otimes}M^{op}\subset L^2(M\bar{\otimes}M^{op})\cong L^2(M)\bar{\otimes}L^2(M)$.

Then for every $t\geqslant 0$, there exists a free family $S_{1}^{(t)},...,S_{n}^{(t)}\in L(\mathbb F_{\infty})$ of  $(0,1)$-semicircular elements and a  $*$-homomorphism $\alpha_t:M\rightarrow M*L(\mathbb F_{\infty})$,  such that $$\|\frac{1}{\sqrt{t}}(\alpha_t(x)-x)-\sum_{i=1}^n\delta_i(x)\#S_{i}^{(t)}\|_2\rightarrow 0,\;\;\text{as}\;\; t\rightarrow 0,\;\;\;\text{for all}\;\;\; x\in M_0.$$  
\end{theorem}

 \section{$L^2$-rigidity results}
 
 The main goal of this section is to prove Theorem \ref{L^2-rigid}. Let us first recall Peterson's notion of $L^2$-rigidity for von Neumann algebras (see \cite[Definition 4.1 and Lemma 2.1]{Pe06}). 
 
 \begin{definition}\cite{Pe06}
 A tracial von Neumann algebra $(M,\tau)$ is  {\it L$^2$-rigid} if for any densely defined real closable derivation $\delta: D(\delta)\rightarrow (L^2(M)\bar{\otimes} L^2(M))^{\oplus\infty}$, the deformation $\phi_t=\exp(-t{\delta}^*\bar{\delta})$ converges uniformly to id$_M$ on $(M)_1$, as $t\rightarrow 0$.
 \end{definition}
 
 \subsection{Proof of Theorem \ref{L^2-rigid}} By \cite[Corollary 4.6]{Pe06} any non-amenable II$_1$ factor that is non-prime or has property Gamma is $L^2$-rigid. Thus, in order to get the conclusion, it suffices to prove that $M$ is not $L^2$-rigid. Assume by contradiction that $M$ is $L^2$-rigid.
 
 Recall that  $\delta:M_0\rightarrow L^2(M)\bar{\otimes} L^2(M)$ is a densely defined real derivation such that $M_0$ contains a non-amenability set $S$ for $M$. 
 For every $t>0$, denote $\phi_t=\exp(-t\delta^*\bar{\delta})$. 
 By Theorem \ref{dilation} there exist a tracial von Neumann algebra $\tilde M$ containing $M$ and $*$-homomorphisms $\alpha_t:M\rightarrow\tilde M$ such that $\phi_t=E_M\circ\alpha_t$, for all $t>0$. 
 
 Since $M$ is  $L^2$-rigid, $\phi_t$ converges uniformly to id$_M$ on $(M)_1$. Thus, we can find  $t_0>0$ such that $\|\phi_t(x)-x\|_2\leqslant\frac{1}{2}$, for all $t\in [0,t_0]$ and every $x\in (M)_1$. Fix $t\in [0,t_0]$.
 Then for every $u\in\mathcal U(M)$ we have that $\tau(\alpha_t(u)u^*)=\tau(\phi_t(u)u^*)\geqslant\frac{1}{2}.$ Denote by $K\subset (\tilde M)_1$ the $\|.\|_2$-closure of the convex hull of the set $\{\alpha_t(u)u^*|u\in\mathcal U(M)\}$ and let $v_t\in K$ be the unique element of minimal $\|.\|_2$. Then $\tau(v_t)\geqslant\frac{1}{2}$, hence $v_t\not=0$, and $\alpha_t(u)v_t=v_tu$, for all $u\in\mathcal U(M)$.
 
 Moreover, if we let $M_t$ be the von Neumann algebra generated by $\alpha_t(M)$ and $M$, then $v_t\in M_t$ and $v_t^*v_t\in M'\cap M_t$.
Thus, we get that  $v_t^*v_t-E_M(v_t^*v_t)\in M'\cap M_t$. On the other hand, Theorem \ref{dilation} gives that the $M$-$M$ bimodule $L^2(M_t)\ominus L^2(M)$ is isomorphic to a sub-bimodule of  $(L^2(M)\bar{\otimes} L^2(M))^{\oplus\infty}$.
Since $M$ is diffuse, we conclude that  $v_t^*v_t\in M'\cap M=\mathbb C1$. Thus, by rescaling $v_t$, we get that there is a unitary $v_t\in M_t$ such that $\alpha_t(M)\subset v_tMv_t^*$, for every $t\in [0,t_0]$. 

Let $t_n\in (0,t_0]$ be a sequence such that $t_n\rightarrow 0$.
Since $S$ is a non-amenability set for $M$,  by  Lemma \ref{uniform} (2) we can find a constant $C>0$ such that for all $n\geqslant 1$ we have $$\|\alpha_{t_n}(x)-\phi_{t_n}(x)\|_2=\|\alpha_{t_n}(x)-E_M(\alpha_{t_n}(x))\|_2\leqslant C\sum_{y\in S}\|\alpha_{t_n}(y)-y\|_2,\;\;\text{for all}\;\; x\in (M)_1.$$

Now, if we take $x\in M_0$, then Lemma \ref{estimate} implies that $\frac{1}{\sqrt{t}}\|\alpha_t(x)-\phi_t(x)\|_2\rightarrow\sqrt{2}\|\delta(x)\|_2$ and $\frac{1}{\sqrt{t}}\|\alpha_t(x)-x\|_2\rightarrow \sqrt{2}\|\delta(x)\|_2$, as $t\rightarrow 0$. In combination with the last inequality this gives that $\|\delta(x)\|_2\leqslant C\sum_{y\in S}\|\delta(y)\|_2$, for all $x\in M_0$. Thus, $\delta$ is bounded,  which is a contradiction.\hfill$\square$

\begin{remark}\label{dense}
Let $M$ be a II$_1$ factor and $p_n\in M$  a  sequence of projections such that $\sum_{n=1}^{\infty}p_n=1$. Then $M_0=\cup_{n\geqslant 1}(p_1+...+p_n)M(p_1+...+p_n)$ is a weakly dense $*$-subalgebra of $M$. Let $\alpha_n$ be a sequence of positive real numbers such that $\sum_{n=1}^{\infty}\alpha_n^2\tau(p_n)^2=+\infty$. Then the map $\delta:M_0\rightarrow L^2(M)\bar{\otimes}L^2(M)$ given by $\delta(x)=\sum_{n=1}^{\infty}i\alpha_n\hskip 0.02in [x,p_n\otimes p_n]$ is a well-defined derivation. Moreover, it is easy to see that $\delta$ is real, unbounded and closable.
This shows that the assumption that $M_0$ contains a non-amenability set for $M$ is necessary in the hypothesis of Theorem \ref{L^2-rigid}.
\end{remark}

\subsection{Proof of Corollary \ref{conjvar}}  We prove here Corollary \ref{conjvar} under the first assumption, and postpone dealing with the second assumption until Corollary \ref{liber}. Denote by $M_0$ the $*$-algebra generated by $X_1,...,X_n$. For every $i\in\{1,...,n\}$, let $\delta_i:M_0\rightarrow L^2(M)\bar{\otimes}L^2(M)$ be the partial free difference quotient derivation given by $\delta_i(X_j)=\delta_{i,j}1\otimes 1$. Further, let $\delta=(\delta_1,...,\delta_n):M_0\rightarrow (L^2(M)\bar{\otimes}L^2(M))^{\oplus n}$.
Since $\delta_i^*(1\otimes 1)=\mathscr J_1(X_i:\mathbb C\langle X_1,...,\hat{X_i},...,X_n\rangle)$ exists and belongs to $L^2(M)$,  \cite[Corollary 4.1]{Vo98} implies that $\delta_i$ is a closable derivation, for all $i\in\{1,2,...,n\}$. Therefore, $\delta$ is  real and closable. 

By \cite[Theorem 13]{Da08}, $M$ is a II$_1$ factor without property Gamma.
Moreover, since the first and second variables are bounded, \cite[Lemmas 9 and 10]{Da08}  imply that $S=\{X_1,...,X_n\}$ is a non-Gamma set for $M$ (see  \cite[Remark 11]{Da08}). By Lemma \ref{non-amen} we therefore deduce that $S$ is a non-amenability set for $M$. 

Since $\Phi^*(X_1,...,X_n)=\sum_{i=1}^n\|\delta_i^*(1\otimes 1)\|_2^2<\infty$, the distribution of $X_i$ has no atoms, for every $i$.
Indeed, by  \cite[Proposition 7.9]{Vo98}, the free entropy $\chi^*$ satisfies $\chi^*(X_1,...,X_n)>-\infty$. Since $\chi^*(X_1)+...+\chi^*(X_n)\geqslant\chi^*(X_1,...,X_n)$ (see \cite[Proposition 7.3]{Vo98}), we deduce that $\chi^*(X_i)>-\infty$, for all $i$. Finally, since $\chi^*(X_i)=\chi(X_i)$ (see \cite[Proposition 7.6]{Vo98}), by \cite[Proposition 4.5]{Vo94}  we get that the distribution of $X_i$ has no atoms, for all $i\in\{1,...,n\}$.

 As in \cite[Section 1.6]{Pe04} it follows that $\delta$ is not inner. Moreover, \cite[Theorem 2.2]{Pe04} implies that $\delta$ is unbounded.
Altogether, we can apply Theorem \ref{L^2-rigid} and deduce the conclusion.\hfill$\square$

\vskip 0.05in 

Let $(M,\tau)$ be a tracial von Neumann algebra generated by $n\geqslant 2$ algebraically free $*$-subalgebras $A_1,...,A_n$.
The next result uses the liberation Fisher information $\varphi^*(A_1;...;A_n)$ introduced in  \cite[Definition 9.1]{Vo99}. The definition involves the derivations $\delta_{A_i}:D(\delta_{A_i})\to L^2(M)\bar{\otimes}L^2(M)$, where $D(\delta_{A_i})$ is the algebra generated by $A_1,...,A_n$, $\delta_{A_i}(a)=a\otimes 1-1\otimes a$, if $a\in A_i,$ and $\delta_{A_i}(a)= 0$, if $a\in A_j$ for $j\neq i$. See \cite[Section 5]{Vo99}  for more details. 

Assuming that $\varphi^*(A_1;...;A_n)=\sum_{i=1}^n||\delta_{A_i}^*(1\otimes1)||_2^2<\infty$, Voiculescu proved  that $\delta_{A_i}$ is a closable derivation \cite[Corollary 6.3]{Vo99}. Note that the combination of Remark 9.2 (f), Proposition 5.9 (a) and Definition 9.1 from \cite{Vo99} gives that \begin{equation}\label{liberation}\varphi^*(A_1;W^*(A_2,...,A_n))=\varphi^*(A_1;\mathbb C\langle A_2,...,A_n\rangle)\leq 2\varphi^*(A_1;...;A_n)\end{equation}

Finally, we will also need the following relation  between Fisher information  and liberation Fisher information: if $\Phi^*(X_1,...,X_n)<\infty$ then $\varphi^*(W^*(X_1);...;W^*(X_n))<\infty$ \cite[Corollary 5.11]{Vo99}.

\begin{corollary}\label{liber}
Let $(M,\tau)$ be a tracial von Neumann algebra which is generated by $n\geqslant 2$ non-trivial ($\neq\mathbb C1$) von Neumann subalgebras $A_1,...,A_n$. Assume that $A_1,...,A_n$  have finite liberation Fisher information $\varphi^*(A_1;...;A_n)<\infty$,  $A_1$ is diffuse,  and $A_2$ is a  non-amenable II$_1$ factor.

Then M is a non $L^2$-rigid II$_1$ factor. Thus, $M$ is prime, does not have property $\Gamma$ nor property $(T)$. In particular, this is the case if $M$  is generated by $m\geqslant 3$ self-adjoint elements $X_1,...,X_m$ satisfying  $\Phi^*(X_1,...,X_m)<\infty$.
\end{corollary}

\begin{proof}
Since $A_1$ is diffuse, by arguing as in \cite[Theorem 1]{Da08} it follows that $M$ is a factor. Indeed, let $x\in\mathcal Z(M)$ and  define $\eta_{\alpha}:M\rightarrow M$ by $\eta_\alpha={\alpha}({\alpha+\delta_{A_2}^*\delta_{A_2}})^{-1}$, for $\alpha>0$. If $y\in A_1$, then $\delta_{A_2}(y)=0$, hence $\eta_{\alpha}(y)=y$. This implies that $\eta_{\alpha}(yx)=y\eta_{\alpha}(x)$ and $\eta_{\alpha}(xy)=\eta_{\alpha}(x)y$. In particular, we get that $[\delta_{A_2}(\eta_\alpha(x)),y]=\delta_{A_2}(\eta_\alpha([x,y]))=0$, for all $\alpha>0$ and $y\in A_1$. Since $A_1$ is diffuse and $\delta_{A_2}(\eta_{\alpha}(x))$ can be viewed as a Hilbert-Schmidt operator, we deduce that $\delta_{A_2}(\eta_{\alpha}(x))=0$. Since $\|\eta_{\alpha}(x)-x\|_2\rightarrow 0$, as $\alpha\rightarrow\infty$, and $\delta_{A_2}$ is closable, we get that $x\in D(\bar{\delta}_{A_2})$ and $\bar{\delta}_{A_2}(x)=0$. 
Thus, for every $y\in A_2$, we have $0=\bar{\delta}_{A_2}([x,y])=[x,y\otimes 1-1\otimes y]=[y,[x,1\otimes 1]]$. Now, since $A_2$ is diffuse we get that $[x,1\otimes 1]=0$, which implies that $x\in\mathbb C1$, as claimed. 
 
Since $A_2$  is a non-amenable II$_1$ factor, by \cite{Co76} we can find a non-amenability set $S\subset A_2$. But then $S$ is a non-amenability set for $M$ which is contained in  $D(\delta_{A_2})$. Since $A_1$ is diffuse, $\delta_{A_2}$ is not inner and hence in not bounded by \cite[Theorem 2.2]{Pe04}. Thus, we can apply Theorem \ref{L^2-rigid} to derive the conclusion.

Now, assume that $M$  is generated by $m\geqslant 3$ self-adjoint elements with  $\Phi^*(X_1,...,X_m)<\infty$. Let $A_1$ be the von Neumann algebra generated by $X_1$ and $A_2$ the von Neumann algebra generated by $X_2,...,X_m$. Since $\Phi^*(X_1,...,X_m)<\infty$, as in the proof  of Corollary \ref{conjvar}, we deduce that the distributions of $X_1,...,X_m$ have no atoms. As a consequence,  $A_1$ and $A_2$ are diffuse. Moreover, by \cite[Theorem 13]{Da08}, $A_2$ is a non-amenable II$_1$ factor. 

Since  $\Phi^*(X_1,...,X_m)<\infty$, we have that  $\varphi^*(W^*(X_1);...;W^*(X_n))<\infty$.
By using equation \ref{liberation} we further get that $\varphi^*(A_1;A_2)\leqslant 2\varphi^*(W^*(X_1);...;W^*(X_m))<\infty$. Altogether, the above shows that $M$ is non $L^2$-rigid factor.
\end{proof}

\section{Lipschitz conjugate variables and absence of Cartan subalgebras}

This section is mainly devoted to the proof Theorem \ref{freedef}.

\subsection{Proof of Theorem \ref{freedef}} We  claim that the Lipschitz condition implies  that the first and second conjugate variables are bounded. Firstly,  Voiculescu (see equation (1) in \cite{Da08}) noticed that for all $x\in M_0=\mathbb C\langle X_1,...,X_n\rangle$ we have \begin{equation}\label{voi}\sum_{i=1}^n(\delta_i(x)(X_i\otimes 1)-(1\otimes X_i)\delta_i(x))=x\otimes 1-1\otimes x.\end{equation}

Moreover, this identity holds for every $x\in D(\bar{\delta})$.
Thus, if $x\in D(\bar{\delta})$ satisfies $\bar{\delta}(x)\in (M\bar{\otimes}M^{op})^{\oplus_n}$, then $x=\tau(x)+E_{M\otimes 1}(\sum_{i=1}^n(\delta_i(x)(X_i\otimes 1)-(1\otimes X_i)\delta_i(x)))$ and therefore $x\in M$. This implies that $\mathscr J_1(X_i:\mathbb C\langle X_1,....,X_{i-1},X_{i+1},...,X_n\rangle=\xi_i=\delta_i^*(1\otimes 1)\in M$, for every $i\in\{1,...,n\}$.
Now, recall that $\mathscr J_2(X_i:\mathbb C\langle X_1,...,X_{i-1},X_{i+1},...,X_n\rangle)=\delta_i^*(\xi_i\otimes 1)$. Since by the second formula in \cite[Proposition 4.1]{Vo98} (applied to $a=\xi_i$ and $\eta=1\otimes 1$) we have that $\delta_i^*(\xi_i\otimes 1)=\xi_i^2-(\tau\otimes\text{id})(\bar{\delta}_i(\xi_i))$, we conclude that $\mathscr J_2(X_i:\mathbb C\langle X_1,...,X_{i-1},X_{i+1},...,X_n\rangle)\in M$, for all $i\in\{1,...,n\}$. This proves our claim.

Since $\Phi^*(X_1,...,X_n)=\sum_{i=1}^n\|\delta_i^*(1\otimes 1)\|_2^2<\infty$, \cite[Theorem 13]{Da08} implies that $M$ is a II$_1$ factor without property Gamma. Moreover, since the first and second order conjugate variables are bounded,  $S=\{X_1,...,X_n\}$ is a non-Gamma set for $M$ (see  \cite[Lemma 9 and 10, and Remark 11]{Da08}).  Lemma \ref{non-amen} gives that $S$ is a non-amenability set for $M$. 

 Let $Q$ be either $\mathbb C1$ or a II$_1$ factor. Our aim is to show that $M\bar{\otimes}Q$ does not have a Cartan subalgebra.
Assume by contradiction that  there is a Cartan subalgebra $A\subset M\bar{\otimes}Q$.
Denote $\tilde M=M*L(\mathbb F_{\infty})$.
 Let $\alpha_t:M\rightarrow \tilde M$ be the $*$-homomorphisms provided by Theorem \ref{free}.  We extend $\alpha_t$ to a $*$-homomorphism $\alpha_t:M\bar{\otimes}Q\rightarrow\tilde M\bar{\otimes}Q$ by letting ${\alpha_t}_{|Q}=\text{id}_Q$. 
 
 The rest of the proof is divided between three claims.
 \vskip 0.05in
 {\bf Claim 1.} Given $t\geqslant 0$, there exist projections $p_t,q_t,r_t\in\mathcal Z(\alpha_t(M)'\cap\tilde M)$ satisfying $p_t+q_t+r_t=1$,
 \begin{enumerate}

 \item  $u_t\alpha_t(M)p_tu_t^*\subset M$, for some unitary $u_t\in\tilde M$,
 \item $v_t\alpha_t(M)q_tv_t^*\subset L(\mathbb F_{\infty})$, for some unitary $v_t\in\tilde M$, 
 \item $\alpha_t(M)r_t\nprec_{\tilde M}M$ and $\alpha_t(M)r_t\nprec_{\tilde M}L(\mathbb F_{\infty})$.
 \end{enumerate}
 
 {\it Proof of Claim 1.} Since $M$ and $L(\mathbb F_{\infty})$ are factors, the sets of projections satisfying (1) and (2) are closed  taking supremum. Let $p_t,q_t\in\mathcal Z(\alpha_t(M)'\cap\tilde M)$ be the maximal projections satisfying (1) and (2).    \cite[Theorem 1.2.1]{IPP05} implies that $p_tq_t=0$. Let $r_t=1-p_t-q_t$. 
If  $\alpha_t(M)r_t\prec_{\tilde M}M$, then, as $M$ is a factor, the proof of \cite[Theorem 5.1]{IPP05} yields a non-zero projection $r\in\mathcal Z(\alpha_t(M)'\cap\tilde M)r_t$ and a unitary $u\in\tilde M$ such that $u\alpha_t(M)ru^*\subset M$. Since $r\leqslant r_t$, this contradicts the maximality of $p_t$. Similarly, $\alpha_t(M)r_t\prec_{\tilde M}L(\mathbb F_{\infty})$ contradicts the maximality of $q_t$.\hfill$\square$

\vskip 0.05in
{\bf Claim 2.} $r_t=0$, for all $t\geqslant 0$.

{\it Proof of Claim 2.} Suppose by contradiction that $r_t\not=0$, for some $t\geqslant 0$.

Since $M$ is a non-amenable II$_1$ factor, by using (3) from Claim 1 and applying \cite[Theorem 6.3]{Io12a} to the inclusion $\alpha_t(M)r_t\subset\tilde M=M*L(\mathbb F_{\infty})$ we get that $(\alpha_t(M)r_t)'\cap (r_tMr_t)^{\omega}\prec_{\tilde M^{\omega}}\mathbb C1$.
 Thus, there exists a non-zero projection $r_t'\in\mathcal Z((\alpha_t(M)r_t)'\cap (r_t\tilde Mr_t)^{\omega})=\mathcal Z(\alpha_t(M)'\cap\tilde M^{\omega})r_t$ such that $r_t'(\alpha_t(M)'\cap\tilde M^{\omega})r_t'$ is completely atomic. By \cite[Lemma 2.7]{Io12a} it follows that $r_t'\in\tilde M$ and moreover $r_t'(\alpha_t(M)'\cap\tilde M^{\omega})r_t'=r_t'(\alpha_t(M)'\cap\tilde M)r_t'$. 
 
 By combining these facts we deduce that there exists a non-zero projection $r_t''\in \alpha_t(M)'\cap\tilde M$ such that $r_t''\leqslant r_t'$ and $r_t''(\alpha_t(M)'\cap\tilde M^{\omega})r_t''=\mathbb Cr_t''$. Thus, we have $(\alpha_t(M)r_t'')'\cap (r_t''\tilde Mr_t'')^{\omega}=\mathbb Cr_t''$.
 
Next, we denote still by $r_t''$ the projection $r_t''\otimes 1\in \tilde M\bar{\otimes}Q$. We define $A_1=\alpha_t(A)r_t''$ and note that $P_1:=\mathcal N_{r_t''(M\bar{\otimes}Q)r_t''}(A_1)''$ contains $(\alpha_t(M)\bar{\otimes}Q)r_t''$. In particular, $P_1$ contains $P_0:=\alpha_t(M)r_t''\otimes 1$. Since  $(\alpha_t(M)r_t'')'\cap (r_t''\tilde Mr_t'')^{\omega}=\mathbb Cr_t''$, by applying Theorem \ref{general} we get that one of the following conditions holds:
(a) $A_1\prec_{\tilde M\bar{\otimes}Q}1\otimes Q$, 
(b) $P_1\prec_{\tilde M}M$ or $P_1\prec_{\tilde M}L(\mathbb F_{\infty})$, or
(c) $P_0$ is amenable.

Now, it is easy to see that (a) implies that $A\prec_{M\bar{\otimes}Q}1\otimes Q$. By taking relative commutants (see e.g. \cite[Lemma 3.5]{Va08}) it follows that $M\otimes 1\prec_{M\bar{\otimes Q}}A$. Since $M$ is non-amenable while $A$ is abelian, this is a contradiction.
By using (3) from Claim 1 and the fact that $M$ is non-amenable, we get that (b) and (c) cannot hold either. This altogether provides the desired contradiction.
\hfill$\square$

\vskip 0.05in
{\bf Claim 3.} $\tau(q_t)\rightarrow 0$, as $t\rightarrow 0$.

{\it Proof of Claim 3.} Since the $M$-$M$ bimodule $L^2(\langle\tilde M,e_{L(\mathbb F_{\infty})}\rangle)$ is isomorphic to $(L^2(M)\bar{\otimes}L^2(M))^{\oplus\infty}$ and $M$ is non-amenable, we get that $M$ is not amenable relative to $L(\mathbb F_{\infty})$ inside $\tilde M$. Since $M$ is a factor, we also have that $M'\cap\tilde M=\mathbb C1$ (see \cite[Remark 6.3]{Po83}). By \cite[Corollary 2.3]{OP07} we derive that for any net of vectors $\xi_n\in L^2(\langle\tilde M,e_{L(\mathbb F_{\infty})}\rangle)$ satisfying $\|x\xi_n\|_2\leqslant \|x\|_2$, for all $x\in \tilde M$, and $\|y\xi_n-\xi_n y\|_2\rightarrow 0$, for all $y\in M$, we must have that $\|\xi_n\|_2\rightarrow 0$.

Let $\xi_t=v_t^*e_{L(\mathbb F_{\infty})}v_tq_t\in L^2(\langle\tilde M,e_{L(\mathbb F_{\infty})}\rangle)$, for any $t\geqslant 0$. Since $v_t\alpha_t(M)q_tv_t^*\subset L(\mathbb F_{\infty})$, we get that $\alpha_t(y)\xi_t=\xi_t\alpha_t(y)$, for all $y\in M$. Also, $\|x\xi_t\|_2,\|\xi_t x\|_2\leqslant \|x\|_2$, for all $x\in\tilde M$ and every $t\geqslant 0$. Thus, if $y\in M$, then  $\|y\xi_t-\xi_ty\|_2\leqslant 2\|y-\alpha_t(y)\|_2$. Since $\|\alpha_t(y)-y\|_2\rightarrow 0$, by using the previous paragraph we conclude that $\|\xi_t\|_2\rightarrow 0$, as $t\rightarrow 0$.
  Since $\|\xi_t\|_2^2=\tau(q_t)$, we are done.\hfill$\square$
  
  We are now ready to derive a contradiction. Recall that $S$ is a non-amenability set for $M$ and the $M$-$M$ bimodule $L^2(\tilde M)\ominus L^2(M)$ is isomorphic to $(L^2(M)\bar{\otimes}L^2(M))^{\oplus\infty}$.  Since $u_{t}\alpha_{t}(M)p_{t}u_{t}^*\subset M$, Lemma \ref{uniform} (2) implies that we can find $C>0$ such that for all $t\geqslant 0$ we have 
  \begin{equation}\label{der} \|(\alpha_t(x)-E_M(\alpha_t(x)))p_t\|_2\leqslant C\sum_{y\in S}\|(\alpha_t(y)-y)p_t\|_2,\;\;\text{for all}\;\; x\in (M)_1.\end{equation}
  
  Now, let $M_0$ be the $*$-algebra generated by $\{X_1,...,X_n\}$ and fix $x\in M_0$ with $\|x\|\leqslant 1$. Theorem \ref{free} gives that $\|\frac{1}{\sqrt{t}}(\alpha_t(x)-x)-\sum_{i=1}^n\delta_i(x)\#S_{i}^{(t)}\|_2\rightarrow 0$.  Note also that $E_M(\delta_i(x)\#S_{i}^{(t)})=0$ and $\|\sum_{i=1}^n\delta_i(x)\#S_{i}^{(t)}\|_2=\sqrt{\sum_{i=1}^n\|\delta_i(x)\|_2^2}=\|\delta(x)\|_2$.  Claims 2 and 3 together imply that $\|p_t-1\|_2\rightarrow 0$. By combining all of these facts and \ref{der} we deduce that $\|\delta(x)\|_2\leqslant C\sum_{y\in S}\|\delta(y)\|_2$. Since $x\in (M_0)_1$ is arbitrary, we deduce that $\delta$ is bounded. As in the proof of Corollary \ref{conjvar}, this leads to a contradiction. \hfill$\square$

\section{Primeness and absence of Cartan subalgebras for regularized algebras}\label{reg}
In this section we establish indecomposibility results for algebras obtained by free additive convolution and liberation. Firstly, for free additive convolution by semicircular variables $\{S_1,...,S_n\}$, we prove that algebras of the form $M_{\varepsilon}=\{X_1+\varepsilon S_1,...,X_n+\varepsilon S_n\}''$ are prime and do not have Cartan subalgebras. More precisely, we have
\begin{theorem}\label{epsilon}
Let $(M,\tau)$ be a tracial von Neumann algebra and $X_1,...,X_n\in M$ be $n\geqslant 2$ self-adjoint elements. Let $\{S_1,...,S_n\}\in L(\mathbb F_n)$ be the canonical semicircular family and $\varepsilon>0$. Denote by $M_{\varepsilon}\subset M*L(\mathbb F_n)$ the von Neumann subalgebra generated by $X_1+\varepsilon S_1,...,X_n+\varepsilon S_n$.

Then $M_{\varepsilon}$ is a non-L$^2$-rigid II$_1$  factor that does not have a Cartan subalgebra.

\end{theorem}

\begin{proof} Fix $\varepsilon>0$.
Then by \cite[Corollary 3.9]{Vo98} we have that $\mathscr J_p (X_i^{\varepsilon}:\mathbb C\langle X_1^{\varepsilon},...,X_{i-1}^{\varepsilon},X_{i+1}^{\varepsilon},...,X_n^{\varepsilon}\rangle)$ exists and belongs to $M_{\varepsilon}$, for all $p\in\{1,2\}$ and $i\in\{1,...,n\}$. By  Corollary \ref{conjvar} it follows that $M_{\varepsilon}$ is a non-$L^2$-rigid II$_1$ factor. 
   
Now, assume by contradiction that $M_{\varepsilon}$ has a Cartan subalgebra and denote $\tilde M=M*L(\mathbb F_n)$.   After enlarging $M$ if necessary (e.g. by replacing it with $M*L(\mathbb F_2)$) we may assume that it is a factor.
 Since $M_{\varepsilon}$ is a non-amenable factor and $M,L(\mathbb F_n)$ are factors, by \cite[Corollary 9.1]{Io12a} we can find unitary elements $u,v\in\tilde M$ and projections $p,q\in\mathcal Z(M_{\varepsilon}'\cap\tilde M)$  such that $uM_{\varepsilon}pu^*\subset M$, $vM_{\varepsilon}qv^*\subset L(\mathbb F_n)$ and $p+q=1$. Since $L(\mathbb F_n)$ is strongly solid \cite{OP07}, we must have that $q=0$, hence $uM_{\varepsilon}u^*\subset M$.

Towards a contradiction, let $i\in\{1,..,n\}$. We define $A_i,M_i$ and $N_i$ to be the von Neumann subalgebras of $\tilde M$ generated by $\{X_i+\varepsilon S_i\}$, $M\cup\{S_i\}$ and $\{S_j\}_{j\in\{1,...,n\}\setminus\{i\}}$, respectively. 
Following \cite[Proposition 4.7]{Vo93}, the distribution of $X_i+\varepsilon S_i$ is absolutely continuous with respect to the Lebesgue measure on $\mathbb R$, hence it has no atoms. This implies that $A_i$ is diffuse.

Since $uA_iu^*\subset M\subset M_i$, $A_i\subset M_i$ and $\tilde M=M_i*N_i$ , by applying \cite{Po83} or \cite[Theorem 1.2.1]{IPP05} we derive that $u\in M_i$.
Since $i\in\{1,...,n\}$ was arbitrary, we conclude that $u\in\cap_{i=1}^nM_i$. Now, freeness easily yields that $\cap_{i=1}^nM_i=M$ and therefore $u\in M$. Thus, we would get that $M_{\varepsilon}\subset M$. This would imply that $S_i\in M$, for all $i\in\{1,...,n\}$, thus giving a contradiction.
\end{proof}

 Our techniques allow us to more generally handle the case of regularization by variables $\{Y_1,...,Y_n\}$ that have bounded first and second order conjugate variables.
  
\begin{theorem}\label{FAC}
Let $(M_1,\tau_1)$, $(M_2,\tau_2)$ be tracial von Neumann algebras and let $M=M_1*M_2$. Let
 $X_1\in M_1\setminus\mathbb C1,X_2,...,X_n\in M_1$ and $Y_1,...,Y_n\in M_2\setminus\mathbb C1$ be self-adjoint elements, for some $n\geqslant 2$.  Denote by $N\subset M$ the von Neumann subalgebra generated by $Z_1=X_1+ Y_1,...,Z_n=X_n+ Y_n$.
 
Assume that $Y_1,...,Y_n$ have bounded first and second order conjugate variables, i.e. we have that $\mathscr J_p (Y_i:\mathbb C\langle Y_1,...,Y_{i-1},Y_{i+1},...,Y_n\rangle)$ exists and belongs to $M_2$, for all $p\in\{1,2\}$ and  $i\in\{1,...,n\}$.
 
Then $N$ is a non-L$^2$-rigid II$_1$  factor that does not have a Cartan subalgebra.

\end{theorem}

\begin{proof} 
Without loss of generality, we may assume that $M_1$ and $M_2$ are factors, and that $\tau(Y_i)=0$, for all $i\in\{1,...,n\}$. 
Let $M_0\subset M$ be the $*$-subalgebra generated by $M_1$ and $\{Y_1,...,Y_n\}$. We define two real closable derivations $\delta_1,\delta_2:M_0\rightarrow L^2(M)\bar{\otimes}L^2(M)$, by letting $\delta_1(x)=0$ and $\delta_2(x)=i[x,1\otimes 1]$, if $x\in M_1$, and $\delta_1(y)=i[y,1\otimes 1]$ and $\delta_2(y)=0$, if $y\in M_2$. 

We continue by proving three facts.

Firstly, let $i\in\{1,...,n\}$. Since  $Y_1,...,Y_n$ have bounded first and second order conjugate variables, $Y_i$ has finite free entropy. By \cite[Proposition 4.5]{Vo94}, the distribution of $Y_i$ is absolutely continuous with respect to the Lebesgue measure on $\mathbb R$. In particular, the distribution of $Y_i$ has no atoms.

 Secondly,  by \cite[Proposition 3.7]{Vo98} we have that the conjugate variables
$$\mathscr J_p (Z_i:\mathbb C\langle Z_1,...,Z_{i-1},Z_{i+1},...,Z_n\rangle)=E_N(\mathscr J_p (Y_i:\mathbb C\langle Y_1,...,Y_{i-1},Y_{i+1},...,Y_n\rangle))$$ exist and belong to $N$, for every $p\in\{1,2\}$, $i\in\{1,..,n\}$. By combining  \cite[Remark 11]{Da08} and Lemma \ref{non-amen}
   we get that $F=\{Z_1,...,Z_n\}$ is a non-Gamma, hence a non-amenability set for $N$. 
   
  Thirdly, let us prove that $\cap_{i=1}^nL^2(\{M_1,Y_i\}'')=L^2(M_1).$ 
We start by considering the free difference quotient $\partial_i:M_0\rightarrow L^2(M)\bar{\otimes}L^2(M)$  with respect to $Y_i$ given by $\partial_i(x)=0$, for $x\in M_1$, and $\partial_i(Y_j)=\delta_{i,j}1\otimes 1$, for $j\in\{1,...,n\}$. By    \cite[Proposition 3.6]{Vo98}, $\partial_i^*(1\otimes 1)$ exists and is equal to $\mathscr J_1 (Y_i:\mathbb C\langle Y_1,...,Y_{i-1},Y_{i+1},...,Y_n\rangle))$. By
\cite[Corollary 4.2]{Vo98} we deduce that  $M_0\otimes M_0\subset D(\partial_i^*)$.

Now, let $Z\in \cap_{i=1}^nL^2(\{M_1,Y_i\}'')$. Towards proving that $Z\in L^2(M_1)$, take a sequence $Z_n\in M_0$ such that $\|Z_n-Z\|_2\rightarrow 0$. If $i\in\{1,...,n\}$, then  $Z\in D(\overline{\partial}_i)$ and $\overline{\partial}_i(Z)=0$.  Since $M_0\otimes M_0\subset D(\partial_i^*)$ we deduce that $\langle\partial_i(Z_n),\zeta\rangle\rightarrow 0$, for all $\zeta\in M_0\otimes M_0$.
Also, a variant of an identity due to Voiculescu (see  \cite[equation (1)]{Da08}) gives  that $\delta_1(Z_n)=i\sum_{j=1}^n[\partial_j(Z_n)(X_j\otimes 1)-(1\otimes X_j)\partial_j(Z_n)].$

We therefore conclude that $\langle\delta_1(Z_n),\zeta\rangle\rightarrow 0$, for all $\zeta\in M_0\otimes M_0$. This implies that $Z\in D(\overline{\delta}_1)$ and $\overline{\delta}_1(Z)=0.$ From this it is easy to see that $Z\in L^2(M_1).$
 
 \vskip 0.05in
Suppose that  $N$ either has a Cartan subalgebra or is $L^2$-rigid. Towards getting a contradiction, we  first use the fact that $F$ is contained in the domains of $\delta_1$ and $\delta_2$ to prove the following:
   
\vskip 0.05in
{\bf Claim 1}. There exists $\xi\in L^2(M)\bar{\otimes}L^2(M)$ such that $\delta_1(x)=[x,\xi]$, for all $x\in N\cap M_0$.
\vskip 0.05in
{\it Proof of Claim 1.}
Let $\tilde{M}=M*L(\mathbb Z)$ and $s\in L(\mathbb Z)$ be a generating $(0,1)$ semicircular element. For $t\in\mathbb R$ and $j\in\{1,2\}$, we define $\alpha_t^{(j)}\in\text{Aut}(\tilde M)$ by letting  $\alpha_t^{(j)}(x)=x$, if $x\in M_j$, and  $\alpha_t^{(j)}(y)=\exp(its)y\exp(-its)$, if $y\in M_k*L(\mathbb Z)$, where $k$ is such that $\{j,k\}=\{1,2\}$. Then it is easy to see (e.g., by combining Theorem \ref{dimagen} and Remark \ref{madef}) that  \begin{equation}\label{ddelta} \|\frac{1}{t}(\alpha_t^{(j)}(x)-x)-\delta_j(x)\#s\|_2\rightarrow 0,\;\;\;\text{as}\;\; t\rightarrow 0,\;\;\;\text{for all}\;\;\; x\in M_0.\end{equation}

To get the conclusion, we treat separately the two cases. Asssume first that $N$ has a Cartan subalgebra. 
Since $N$ is a non-amenable factor and $M_1,M_2$ are factors, \cite[Corollary 9.1]{Io12a} implies that we can find unitary elements $u_1,u_2\in M$ and projections $p_1,p_2\in\mathcal Z(N'\cap M)$  such that $u_1Np_1u_1^*\subset M_1$, $u_2Np_2u_2^*\subset M_2$ and $p_1+p_2=1$.

Let $t\in\mathbb R$ and $i\in\{1,2\}$. 
Then $\alpha_t^{(i)}(u_i)\alpha_t^{(i)}(N)\alpha_t^{(i)}(p_i)\alpha_t^{(i)}(u_i)^*\subset\alpha_t^{(i)}(M_i)\subset M$.  Also, the $M$-$M$ bimodule $L^2(\tilde{M})\ominus L^2(M)$ is isomorphic to $(L^2(M)\bar{\otimes}L^2(M))^{\oplus\infty}$. Since $F$ is a non-amenability set for $N$, Lemma \ref{uniform} (2) provides a constant $C>0$ such that for all $x\in (N)_1$  we have \begin{equation}\label{alpha}\|(\alpha_t^{(i)}(x)-E_{M}(\alpha_t^{(i)}(x)))\alpha_t^{(i)}(p_i)\|_2\leqslant C\sum_{y\in F}\|\alpha_t^{(i)}(y)-y\|_2. \end{equation}

Note that $E_{M}(\delta_i(x)\#s)=0$, for all $x\in M_0$. Thus, combining equations \ref{ddelta} and \ref{alpha} yields that $\|\delta_i(x)p_i\|_2\leqslant C\sum_{y\in F}\|\delta(y)\|_2$, for all $x\in (N\cap M_0)_1$.
Since $\delta_1(x)p_2=i[x,1\otimes 1]p_2-\delta_2(x)p_2$, we get that $\|\delta_1(x)p_2\|_2\leqslant 2+ C\sum_{y\in F}\|\delta(y)\|_2$, for all $x\in (N\cap M_0)_1$.  Since $p_1+p_2=1$, it follows that the restriction of $\delta_1$ to $N\cap M_0$ is bounded. The conclusion follows as in the proof of \cite[Theorem 2.2]{Pe04}.

Now, assume that $N$ is $L^2$ rigid.  Note that the restriction of $\delta_1$ to $N\cap M_0$ is a real closable derivation into $L^2(M)\bar{\otimes}L^2(M)$. Since $L^2(M)\bar{\otimes}L^2(M)$ is isomorphic to a $N$-$N$ sub-bimodule of $(L^2(N)\bar{\otimes}L^2(N))^{\oplus\infty}$ and $N\cap M_0$ contains a non-amenability set for $N$,  Theorem \ref{L^2-rigid} implies that $\delta_1$ is bounded on $N\cap M_0$. The conclusion now follows as above. \hfill$\square$

Now, by the definition of $\delta_1$ we have that $\delta_1(Z_j)=i[Y_j,1\otimes 1]$. Denote $\eta=-i\xi\in L^2(M)\bar{\otimes}L^2(M)$. Since Claim 1 yields that $\delta_1(Z_i)=[Z_i,\xi]$, we conclude that \begin{equation}\label{Yi} [Y_i,1\otimes 1]=[Z_i,\eta],\;\;\;\text{for all}\;\;\;i\in\{1,...,n\}.\end{equation}
 
Next, we identify $L^2(M)\bar{\otimes}L^2(M)$ with $L^2(M\bar{\otimes}M^{op})$ in the natural way such that the $M$-$M$ bimodule structure of $L^2(M)\bar{\otimes}L^2(M)$ corresponds to the left multiplication action of $M\bar{\otimes}M^{op}$ on $L^2(M\bar{\otimes}M^{op})$. We can therefore rewrite \ref{Yi} as \begin{equation}\label{yi}Y_i\otimes 1-1\otimes Y_i^{op}=(Z_i\otimes 1-1\otimes {(Z_i)}^{op})\eta,\;\;\;\text{for all}\;\;\; i\in\{1,...,n\}.\end{equation}
\vskip 0.05in
{\bf Claim 2}. $\eta\in L^2(M_1)\bar{\otimes}L^2(M_1)\cong L^2(M_1\bar{\otimes}M_1^{op}).$
\vskip 0.05in
{\it Proof of Claim 2.} Fix $i\in\{1,...,n\}$ and denote by $M^{(i)}=\{M_1,Y_i\}''$ the von Neumann algebra generated by $M_1$ and $Y_i$.
Since $X_i\in M_1\setminus\mathbb C1$ and the distribution of $Y_i$  has no atoms, employing \cite[Theorem 4.1]{Be06} we derive that the distribution of $Z_i=X_i+Y_i$ is absolutely continuous with respect to the Lebesgue measure on $\mathbb R$ and thus has no atoms. This implies that the self-adjoint element $T_i:=Z_i\otimes 1-1\otimes {(Z_i)}^{op}\in M\bar{\otimes} M^{op}$ has no kernel. 

For $\delta>0$, let $A_{\delta}=\{r\in\mathbb R|\hskip 0.01in |r|\geqslant\delta\}$ and $f_{\delta}:\mathbb R\rightarrow\mathbb R$ be the Borel function $f_{\delta}(r)=\frac{1}{r}1_{A_{\delta}}(r)$.
Since $T_i$ has no kernel,  $f_{\delta}(T_i)T_i=1_{A_{\delta}}(T_i)\rightarrow $ id, in the strong operator topology, as $\delta\rightarrow 0$. Thus, by using formula \ref{yi} we derive  that $\|f_{\delta}(T_i)(Y_i\otimes 1-1\otimes Y_i^{op})-\eta\|_2=\|(f_{\delta}(T_i)T_i-1)\eta\|_2\rightarrow 0$, as $\delta\rightarrow 0$. 
Since $Y_i\in M^{(i)}$ and $T_i\in M^{(i)}\bar{\otimes}M^{(i)op}$, we get that $f_{\delta}(T_i)
(Y_i\otimes 1-1\otimes Y_i^{op})\in M^{(i)}\bar{\otimes}M^{(i)op}$.

We deduce that  $\eta\in L^2(M^{(i)}\bar{\otimes}M^{(i)op})\cong L^2(M^{(i)})\bar{\otimes}L^2(M^{(i)})$, for all $i\in\{1,...,n\}$. Since we proved that  $\cap_{i=1}^nL^2(M^{(i)})=L^2(M_1)$, we conclude that $\eta\in L^2(M_1)\bar{\otimes}L^2(M_1)$, as claimed.\hfill$\square$

Let $P$ be the orthogonal projection from $L^2(M)\bar{\otimes}L^2(M)$ onto $L^2(M)\bar{\otimes}L^2(M_1)$. 
Equation \ref{Yi}  gives that $Y_1\otimes 1-1\otimes Y_1=Z_1\eta-\eta Z_1$. Since $E_{M_1}(Y_1)=0$ and $\eta\in L^2(M_1)\bar{\otimes}L^2(M_1)$, by applying $P$ to the last identity, we deduce that 
$Y_1\otimes 1=Z_1\eta-\eta X_1$. Hence $1\otimes Y_1= \eta Y_1$. Since $Y_1$ is diffuse, this implies that $\eta=1\otimes 1$ and further that $X_1\otimes 1=1\otimes X_1$.
Thus, $X_1\in\mathbb C1$, which is the desired contradiction.\end{proof}

Finally, we prove an indecomposability result for  regularized algebras obtained by liberation in the sense of \cite[Section 2]{Vo99}. 

\begin{theorem}\label{liberation}
Let $(M_1,\tau_1)$, $(M_2,\tau_2)$ be tracial von Neumann algebras and  $M=M_1*M_2$.  Let $A_1,...,A_n\subset M_1$ be diffuse von Neumann subalgebras and $u_1,...,u_n\in M_2$ be unitary elements, for some $n\geqslant 2$. Denote by $N\subset M$ the von Neumann subalgebra generated by $u_1A_1u_1^*,...,u_nA_nu_n^*$ .

Assume that $A_1$ is a non-amenable II$_1$ factor and that $u_2\not\in\mathbb Cu_1$.

Then $N$ is a non-L$^2$-rigid II$_1$  factor that does not have a Cartan subalgebra.
\end{theorem}

\begin{proof} Let us first show that $N$ is a factor. Let $x\in\mathcal Z(N)$. Then $[u_1^*xu_1,A_1]=[u_2^*xu_2,A_2]=0$. Since $A_1,A_2\subset M_1$ are diffuse,  applying \cite{Po83}  or \cite[Theorem 1.2.1]{IPP05}  gives that $u_1^*xu_1$ and $u_2^*xu_2$ belong to $M_1$. Equivalently, $x\in u_1M_1u_1^*\cap u_2M_1u_2^*$.
Since $u_1,u_2\in M_2$ and $u_2\not\in\mathbb Cu_1$, we get that $x\in\mathbb C1$.

Next, let  $M_0\subset M$ be the $*$-subalgebra generated by $M_1$ and $M_2$. Consider the  real closable derivation $\delta_1:M_0\rightarrow L^2(M)\bar{\otimes}L^2(M)$ by letting $\delta_1(x)=0$, if $x\in M_1$, and $\delta_1(y)=i[y,1\otimes 1]$, if $y\in M_2$. Since $A_1$ is a non-amenable II$_1$ factor, \cite{Co76} implies that there exists a non-amenability set $F\subset A_1$. Thus, $u_1Fu_1^*$ is a non-amenability set for $N$ which is contained in the domain of $\delta_1$.

If $N$ is either $L^2$-rigid or has a Cartan subalgebra, then the proof of Theorem \ref{FAC} implies that there  exists $\xi\in L^2(M)\bar{\otimes}L^2(M)$ such that $\delta_1(x)=[x,\xi]$, for all $x\in N\cap M_0$.

To get a contradiction, let $j\in\{1,...,n\}$.
Since $A_j$ is diffuse, we can find a self-adjoint element $X_j\in A_j$ which generates a diffuse algebra. If we define $Z_j=u_jX_ju_j^*$ then we have  $$[Z_j,\xi]=\delta_1(u_jX_ju_j^*)=i([u_j,1\otimes 1]X_ju_j^*+u_jX_j[u_j^*,1\otimes 1])=i[Z_j,1\otimes 1-u_j\otimes u_j^*].$$

Since $Z_j$ is diffuse, we deduce that $\xi=i(1\otimes 1-u_j\otimes u_j^*)$ for all $j\in\{1,...,n\}$. This implies in particular that $u_1\otimes u_1^*=u_2\otimes u_2^*$, and hence that $u_2\in\mathbb Cu_1$, which gives a contradiction.
\end{proof}

\section{Algebraic derivations and absence of Cartan subalgebras}

The main goal of this section is to establish a general result showing absence of Cartan subalgebras for any II$_1$ factor $M$ admitting certain unbounded ``algebraic" derivations.
 More precisely:

\begin{theorem}\label{nocartan} Let $M$ be a II$_1$ factor, $B\subset M$ be a von Neumann subalgebra and $M_0\subset M$ be a weakly dense $*$-subalgebra, such that $M_0$ contains a non-amenability set for $M$ relative to $B$.  
Assume that for any non-zero projection $r\in B'\cap M$, there exists a mixing $B$-$B$ sub-bimodule $\mathcal H$ of $L^2(M)$ such that $r\mathcal Hr\not=\{0\}$.
Let $D(\delta)\subset M$ be a $*$-subalgebra which contains $M_0$ and $B$.

Assume that   there exists a real derivation
$\delta:D(\delta)\rightarrow L^2(\langle M,e_B\rangle)$ such that  $\delta_{|M_0}$ is unbounded, $\delta^*(e_B)$ exists and belongs to $M_0$, and $\delta(b)=0$, for all $b\in B$.

Also, suppose that $M_0$ is finitely generated and $\delta(M_0)\subset\;\text{span}\; M_0e_BM_0$. More generally, suppose that  $M_0=\cup_{n\geqslant 1}M_n$, where $M_n$ is a finitely generated $*$-subalgebra such that $M_n\subset M_{n+1}$ and  $\delta(M_n)\subset\;\text{span}\; M_ne_BM_n$, for all $n\geqslant 1$.

Then $M$  has no Cartan subalgebra and does not have property Gamma. 

\end{theorem}

The mixingness condition was inspired by \cite[Corollary C]{Ho12}, where it is shown that if an orthogonal representation $\pi:G\rightarrow\mathcal O(H_{\mathbb R})$ contains a mixing subrepresentation, then   the II$_1$ factor $\Gamma(H_{\mathbb R})''\rtimes G$ associated to the corresponding free Bogoljubov action  has no Cartan subalgebra.

Before proceeding to the proof of Theorem \ref{nocartan}, let us derive several consequences of it. Firstly, note that Corollary \ref{cartan} corresponds precisely to the case $B=\mathbb C1$. Secondly, let us deduce Corollary \ref{amalgam}.

\subsection{Proof of Corollary \ref{amalgam}} 
Recall that $M=M_1*_{B}M_2$. 
Let $D(\delta)$ be the $*$-algebra generated by $M_1$ and $M_2$ and define $\delta:D(\delta)\rightarrow L^2(\langle M,e_B\rangle)$ by letting $\delta(x)=i[x,e_B]$, if $x\in M_1$, and $\delta(x)=0$, if $x\in M_2$. Then it is easy to see that $\delta$ is a real derivation and $\delta^*(e_B)=0$. 

Let $M_{1,n}\subset M_1$ and $M_{2,n}\subset M_2$ be increasing sequences of finitely generated $*$-subalgebras such that $M_{1,0}=\cup_{n\geqslant 1}M_{1,n}$ is weakly dense in $M_1$ and $M_{2,0}=\cup_{n\geqslant 1}M_{2,n}$ is weakly dense in $M_2$. Assume that $u\in M_{1,1}$ and $v\in M_{2,1}$. Denote by $M_n$ the algebra generated by $M_{1,n}$ and $M_{2,n}$. Then $M_0=\cup_{n\geqslant 1}M_n$ is weakly dense in $M$ and  $\delta(M_n)\subset\;\text{span}\; M_ne_BM_n$, for all $n\geqslant 1$. Since $u,v$ are unitaries, $u\in M_1\ominus B$ and $v\in M_2\ominus B$, it follows that $\|\delta((uv)^n)\|_2=\sqrt{2n}$, for all $n\geqslant 1$. This shows that $\delta_{|M_0}$ is unbounded. Thus, $\delta$ satisfies all the assumptions required in Theorem \ref{nocartan}. 

We continue by verifying the rest of the  assumptions from Theorem \ref{nocartan}.

We first claim that $M$ is a factor and does not have property Gamma. Let $x\in M'\cap M^{\omega}$.  Since $E_B(u)=E_B(v)=E_B(w)=E_B(w^*v)=0$, \cite[Lemma 6.1]{Io12a} gives  that $M'\cap M^{\omega}\subset B^{\omega}$, and thus  $x\in B^{\omega}$. Recall that $zBz^*\perp B$, for some $z\in\{u,v\}$. Therefore,  $\langle z(x-\tau(x))z^*,x-\tau(x)\rangle=0$. On the other hand, since $x$ commutes with $z$, we have that $z(x-\tau(x))z^*=x-\tau(x)$. Altogether, it follows that $x=\tau(x)\in\mathbb C1$, thereby proving the claim.

Next, using the same argument as in the proof of \cite[Lemma 6.1]{Io12a},  we prove that $S=\{u,v,w\}$ is a non-amenability set for $M$ relative to $B$. For $i\in\{1,2\}$, we denote by $W_i\subset M$ the set of alternating words in $M_1\ominus B$ and $M_2\ominus B$ which start in $M_i\ominus B$, and  by $\mathcal H_i\subset L^2(\langle M,e_B\rangle)$ the $\|.\|_2$-closure of the linear span of $W_ie_BM$. Let $\mathcal H_0\subset L^2(\langle M,e_B\rangle)$ be the $\|.\|_2$-closure of $e_BM$. Then  $L^2(\langle M,e_B\rangle)=\mathcal H_0\oplus \mathcal H_1\oplus \mathcal H_2$.
 For $i\in\{0,1,2\}$, let $e_i$ be the orthogonal projection onto $\mathcal H_i$. 
 
 Notice that if $x\in M_1\ominus B$ and $y\in M_2\ominus B$ then $x\mathcal H_2x^*\subset \mathcal H_1$ and $y\mathcal H_1y^*\subset \mathcal H_2$. Since $E_B(u)=E_B(v)=E_B(w)=E_B(w^*v)=0$, we deduce that \begin{equation}\label{ortho}u\mathcal H_2u^*\subset \mathcal H_1,\;\; v\mathcal H_1v^*\subset \mathcal H_2,\;\; w\mathcal H_1w^*\subset \mathcal H_2\;\;\text{and}\;\; v\mathcal H_1v^*\perp w\mathcal H_1w^*.\end{equation}

Now, let $\xi\in L^2(\langle M,e_B\rangle)$ and denote $C_{\xi}=\sum_{z\in S}\|z\xi z^*-\xi\|_2$. Then equation \ref{ortho} implies that $$\|e_2(u^*\xi u)\|_2\leqslant \|e_1(\xi)\|_2\;\;\text{and}\;\; \|e_1(v^*\xi v)\|_2^2+\|e_1(w^*\xi w)\|_2^2\leqslant \|e_2(\xi)\|_2^2.$$

These inequalities further imply that  $\|e_2(\xi)\|_2-C_{\xi}\leqslant\|e_1(\xi)\|_2$ and $\sqrt{2}(\|e_1(\xi)\|_2-C_{\xi})\leqslant\|e_2(\xi)\|_2$. From this it is easy to derive that $\|e_1(\xi)\|_2\leqslant 6C_{\xi}$ and $\|e_2(\xi)\|_2\leqslant 7C_{\xi}$. Since $u\mathcal H_0u^*\subset \mathcal H_1$, we similarly get that $\|e_0(\xi)\|_2\leqslant \|e_1(\xi)\|_2+C_{\xi}\leqslant 7C_{\xi}$. Altogether, it follows that $\|\xi\|_2\leqslant 20C_{\xi}$, proving that $S$ is indeed a non-amenability set relative to $B$.

Finally,  let $r\in M$ be a non-zero projection. We claim that there exists a mixing $B$-$B$ bimodule $\mathcal H\subset L^2(M)$ such that $r\mathcal Hr\not=\{0\}$. Assume by contradiction that this is false. Let $z_1,z_2\in\{u,v\}$ such that $z_1 Bz_1^*\perp B$ and $\{z_1,z_2\}=\{u,v\}$. Denote $z=z_2z_1$. Then for every $k\geqslant 1$, by using freeness, it is clear that $z^k B {z^k}^*\perp B$. Thus, the $B$-$B$ bimodule $\mathcal H_k=\overline{Bz^kB}^{\|.\|_2}$ is isomorphic to the coarse $B$-$B$ bimodule, $L^2(B)\bar{\otimes} L^2(B)$, and is therefore mixing.

By our assumption we have that $r\mathcal H_kr=\{0\}$, hence $rz^kr=0$, for all $k\geqslant 1$. In particular, we get that $(\frac{1}{n}\sum_{k=1}^n{z^k}^*rz^k)r=0$, for all $n\geqslant 1$. Let $D\subset M$ denote the von Neumann subalgebra generated by $z$. Von Neumann's ergodic theorem implies that $\frac{1}{n}\sum_{k=1}^n{z^k}^*rz^k$ converges in $\|.\|_2$ to $E_{D'\cap M}(r)$, as $n\rightarrow\infty$. Thus, we derive that $E_{D'\cap M}(r)r=0$ and further that $(E_{D'\cap M}(r))^2=0$. Since $r\geqslant 0$ it follows that $E_{D'\cap M}(r)=0$ and  $\tau(r)=\tau(E_{D'\cap M}(r))=0$, hence $r=0$. This provides the desired contradiction.

Altogether, we can apply Theorem \ref{nocartan} and derive the conclusion.\hfill$\square$

In the proof of Theorem \ref{nocartan} we will need the following technical result which says that if $L^2(M)$ contains ``enough" mixing $B$-$B$ bimodules, then no Cartan subalgebra of $M$ can be embedded into $B$. More precisely, we have

\begin{proposition}\label{mixing}
Let $M$ be a II$_1$ factor and $B\subset M$ be a von Neumann subalgebra. Assume that  for any non-zero projection $r\in B'\cap M$, there exists a mixing 
$B$-$B$ sub-bimodule $\mathcal H$ of $L^2(M)$ such that $r\mathcal Hr\not=\{0\}$.

\begin{enumerate}
\item If $A\subset M$ is a Cartan subalgebra, then $A\nprec_{M}B$. 
\vskip 0.05in
\item If $M$ has property $\Gamma$, then $M'\cap M^{\omega}\nprec_{M^{\omega}}B^{\omega}$. Moreover, in this case, let $(N,\tau')$ be any tracial von Neumann algebra containing $B$ such that ${\tau}_{|B}={\tau'}_{|B}$ and denote $\tilde M=M*_{B}N$. Then $M'\cap M^{\omega}\nprec_{\tilde M^{\omega}}B^{\omega}$.
 \end{enumerate}
\end{proposition}

\begin{proof}
 (1) Assume by contradiction that  $A\subset M$ is a Cartan subalgebra such that $A\prec_{M}B$. Then we can find projections $p\in A,q\in B$,  a non-zero partial isometry $v\in M$ and a $*$-homomorphism $\phi:Ap\rightarrow qBq$ such that $v^*v=p$, $q_0=vv^*\leqslant q$ and $\phi(x)v=vx$, for all $x\in Ap$.

 Towards  a contradiction, let $r\in B'\cap M$ be the smallest projection such that $q_0\leqslant r$. By the hypothesis we can find a mixing $B$-$B$ bimodule $\mathcal H\subset L^2(M)$ such that $r\mathcal Hr\not=\{0\}$. Denote by $e$ the orthogonal projection from $L^2(M)$ onto $\mathcal H$. 
 
  Denote $B_0=\phi(Ap)\subset qBq$.
 Fix $u\in\mathcal N_{pMp}(Ap)$ and denote by $\theta$ the automorphism of $B_0$ given by $\theta=\phi\circ\text{Ad}(u)\circ\phi^{-1}$. Then  we have that \begin{equation}\label{inter}vuv^*\; y=\theta(y)\;vuv^*\;\;\;\text{for all}\;\;\; y\in B_0\end{equation}

 Since $B_0$ is diffuse, we can find a sequence $y_n\in\mathcal U(B_0)$ such that $y_n\rightarrow 0$, weakly. 
 Since $\mathcal H$ is a $B$-$B$ bimodule by using equation \ref{inter} we get that $e(vuv^*)y_n=\theta(y_n)e(vuv^*)$, for all $n$. Hence we have that $\langle \theta(y_n^*)e(vuv^*)y_n,e(vuv^*)\rangle=\|e(vuv^*)\|_2^2$, for all $n$. Since $y_n\rightarrow 0$ weakly and $\mathcal H$ is a mixing $B$-$B$ bimodule, we derive that $e(vuv^*)=0$.
 
Since this holds for any unitary $u\in\mathcal N_{pMp}(Ap)$ and $Ap\subset pMp$ is a Cartan subalgebra, we conclude that $e(q_0Mq_0)=e(vpMpv^*)=\{0\}$.
 Let us show that this implies that $e(rMr)=\{0\}$.
 
 Indeed, denote by $\mathcal K$ the $\|.\|_2$ closure of the convex hull of the set $\{wq_0w^*|w\in\mathcal U(B)\}$. Since $e(q_0Mq_0)=\{0\}$ and $e$ is the orthogonal projection onto a $B$-$B$ bimodule, it follows that  we have $e(zMz)=\{0\}$, for all $z\in{\mathcal K}$. Since $E_{B'\cap M}(q_0)\in{\mathcal K}$ (more precisely, $E_{B'\cap M}(q_0)$ is the unique element of minimal $\|.\|_2$ in ${\mathcal K}$) we get that $e(E_{B'\cap M}(q_0)ME_{B'\cap M}(q_0))=\{0\}$.
Since $r$ is equal to the support of $E_{B'\cap M}(q_0)$, we have that $E_{B'\cap M}(q_0)ME_{B'\cap M}(q_0)$ is a $\|.\|_2$-dense subspace of $rMr$.  Thus, we would get that $e(rMr)=\{0\}$ or, equivalently, that $r\mathcal H r=\{0\}$, which provides the desired contradiction.

\vskip 0.05in

(2) First, assume by contradiction that $M$ has property Gamma and that $M'\cap M^{\omega}\prec_{M^{\omega}}B^{\omega}$.

Let $\{y_i\}_{i\geqslant 1}$ be a $\|.\|_2$ dense sequence in $M$.
Since $M$ has property Gamma, by a construction of Popa (see the proof of \cite[Proposition 7]{Oz03}) we can find  diffuse abelian von Neumann subalgebras  $\{A_n\}_{n\geqslant 1}$  of $M$ such that for all $n$ we have that $A_{n+1}\subset A_n$ and that
\begin{equation}\label{comm}\|y_i-E_{A_n'\cap M}(y_i)\|_2\leqslant\frac{1}{n},\;\;\text{for all}\;\;1\leqslant i\leqslant n.\end{equation} 

Then we have

{\bf Claim}. $A_n\prec_{M}B$, for some $n\geqslant 1$.

{\it Proof of the claim}.
Denote by $A_{\omega}=\prod_{n=1}^{\omega}A_n$ the von Neumann subalgebra of $M^{\omega}$ consisting of all $x=(x_n)_n$ such that $\lim_{n\rightarrow\omega}\|x_n-E_{A_n}(x_n)\|_2=0$. 

Then \ref{comm} implies that $A_{\omega}\subset M'\cap M^{\omega}$.
Since $M'\cap M^{\omega}\prec_{M^{\omega}}B^{\omega}$, we get that $A_{\omega}\prec_{M^{\omega}}B^{\omega}$. Thus, we can find projections $p\in A_{\omega}, q\in B^{\omega}$, a non-zero partial isometry $v\in qM^{\omega}r$ and a $*$-homomorphism $\phi:A_{\omega}p\rightarrow qB^{\omega}q$ such that $\phi(x)v=vx$, for all $x\in A_{\omega}r$.

 Let $\delta=\|E_{B^{\omega}}(vv^*)\|_2$. Then $||E_{B^{\omega}}(vuv^*)||_2=\delta$, for all  $u\in\mathcal U(A_{\omega}p).$
Write $p=(p_n)_n$ and $v=(v_n)_n$, where $p_n\in A_n$ is a projection and $v_n\in M$, for all $n$. 

If the claim is false, then $A_n\nprec_{M}B$ and thus $A_np_n\nprec_{M}B$, for all $n\geqslant 1$. Thus, for every $n\geqslant 1$, we can find a unitary $u_n\in A_np_n$ such that $\|E_B(v_nu_nv_n^*)\|_2\leqslant\frac{\delta}{2}$.  Then the unitary $u=(u_n)\in A_{\omega}p$ satisfies $\|E_{B^{\omega}}(vuv^*)\|_2=\lim_{n\rightarrow\omega}\|E_B(v_nu_nv_n^*)\|_2\leqslant\frac{\delta}{2}$, which gives a contradiction.
\hfill$\square$

Let $n$ such that $A_n\prec_{M}B$. Then we can find projections $p\in A_n,q\in B$, a non-zero partial isometry $v\in qMp$ and a $*$-homomorphism $\phi:A_np\rightarrow qBq$ such that $\phi(x)v=vx$, for all $x\in A_np$. Denote $q_0=vv^*\leqslant q$ and let $r\in B'\cap M$ be the smallest projection such that $q_0\leqslant r$. The hypothesis implies the existence of a non-zero mixing $B$-$B$ bimodule $\mathcal H\subset L^2(M)$ such that $r\mathcal Hr\not=\{0\}$. Denote by $e$ the orthogonal projection from $L^2(M)$ onto $\mathcal H$.

Now, let $N\geqslant n$,
 $u\in A_N'\cap M$ and $x\in A_Np$. Write $x=x_0p$, where $x_0\in A_N$. Since $u$ and $x_0$ commute and $v=vp$ we get that $vuv^*\phi(x)=vuxv^*=vux_0v^*=vx_0uv^*=vxuv^*=\phi(x)vuv^*$. This shows that  $v(A_N'\cap M)v^*\subset q_0Mq_0$ commutes with $\phi(A_Np)\subset qBq$. 
 
 Since $\mathcal H$ is a mixing $B$-$B$ bimodule and $A_N$ is diffuse, by repeating the argument from the proof of (1) we get that $e(v(A_N'\cap M)v^*)=\{0\}$, for all $N\geqslant n$. Equation \ref{comm} then implies that $e(vy_iv^*)=0$, for all $i\geqslant 1$. By using the $\|.\|_2$ density of $\{y_i\}_{i\geqslant 1}$ in $M$ we conclude that $e(q_0Mq_0)=e(vMv^*)=\{0\}$
 and the end of the proof of (1) yields a contradiction.
 
 To prove the moreover assertion, assume by contradiction that $M'\cap M^{\omega}\prec_{\tilde M^{\omega}}B^{\omega}$. Then the above claim implies that $A_n\prec_{\tilde M}B$, for some $n\geqslant 1$.
 Since $A_n\subset M$ and $\tilde M=M*_{B}N$,  by \cite[Theorem 1.2.1]{IPP05}  we get that $A_n\prec_{M}B$. Continuing as above yields a contradiction.
\end{proof}

\subsection{Proof of Theorem \ref{nocartan}}   Define $\tilde M=M*_{B}(B\bar{\otimes}L(\mathbb Z)$) and let $\{\alpha_t\}_{t\in\mathbb R}$ be the one-parameter group of automorphisms of $\tilde M$ arising from $\delta$ as provided by Proposition \ref{dimagen}. Note that since $\delta_{|B}\equiv 0$, we have that $\alpha_t(x)=x$, for all $x\in B$ and every $t\in\mathbb R$. 
Let $S\subset M_0$ be a non-amenability set for $M$ relative to $B$. 

The core of the proof consists of proving several claims about the inclusion  $M\subset\tilde M$.
\vskip 0.05in
{\bf Claim 1.} $\delta_{|M_0}$ is unbounded.

{\it Proof of Claim 1.} Let $s\in L(\mathbb Z)$ be a semicircular element and $L^2(\langle M,e_B\rangle)\ni\xi\rightarrow\xi\# s\in L^2(\tilde M)$ be the unique embedding of $M$-$M$ bimodules sending $e_B$ to $s$. Let $a,b,c\in M_0$. Since $\delta$ is a real derivation, a calculation in the spirit of  the proof of \cite[Proposition 4.1]{Vo98} gives that 

$$\langle a\delta^*(e_B)b-E_M((\delta(a)\#s)s)b-aE_M(s(\delta(b)\#s)),c\rangle=$$ $$\langle\delta^*(e_B),a^*cb^*\rangle-\langle\delta(a)\#s,cb^*s\rangle-\langle\delta(b)\#s,sa^*c\rangle=$$ $$\langle\delta(bc^*a)-bc^*\delta(a)-\delta(b)c^*a,e_B\rangle=\langle b\delta(c^*)a,e_B\rangle=\langle ae_Bb,\delta(c)\rangle.$$

Thus, $ae_Bb$ belongs to the domain of $\delta^*$, for all $a,b\in M_0$. Hence, for all $x\in D(\delta)$, we have that 
$$\|\delta(x)\|_2=\sup_{y\in\hskip 0.02in\text{span}(M_0e_BM_0),\hskip 0.02in \|y\|_2\leqslant 1}|\langle \delta(x),y\rangle|=\sup_{y\in\hskip 0.02in\text{span}(M_0e_BM_0),\hskip 0.02in \|y\|_2\leqslant 1}|\langle x,\delta^*(y)\rangle|$$

Since $\delta$ is unbounded and $M_0\subset D(\delta)$ is dense in $\|.\|_2$, it follows that ${\delta}_{|M_0}$ is unbounded. \hfill$\square$

\vskip 0.05in
{\bf Claim 2.} $M'\cap \tilde M^{\omega}\subset M^{\omega}$.

{\it Proof of Claim 2.} 
Since $\tilde M=M*_{B}(B\bar{\otimes}L(\mathbb Z))$, there exists a $B$-$M$ bimodule $\mathcal K$ such that as $M$-$M$ bimodules we have $L^2(\tilde M)\ominus L^2(M)\cong L^2(M)\otimes_{B}\mathcal K$. Since $S$ is a non-amenability set for $M$ relative to $B$,   the second part of Lemma \ref{gap} implies that there is $\kappa>0$ such that $\|\xi\|_2\leqslant \kappa\sum_{y\in S}\|y\xi-\xi y\|_2$, for all $\xi\in L^2(\tilde M)\ominus L^2(M)$. This gives that $M'\cap \tilde M^{\omega}\subset M^{\omega}$.
\hfill$\square$

 Note that since  $M$ is a factor, Claim 2 implies that $M'\cap \tilde M=\mathbb C1$.

{\bf Claim 3.}  $\alpha_t(M)$ is not amenable relative to $B$ inside $\tilde M$, and $\alpha_t(M)\nprec_{\tilde M}B\bar{\otimes}L(\mathbb Z)$, for any $t\in\mathbb R$.
 
 {\it Proof of Claim 3.} Consider the $B$-$M$ bimodule $\mathcal H=\mathcal K\otimes_BL^2(\tilde M)$, where $\mathcal K$ is as in the proof of Claim 2. Then we have that $L^2(\langle\tilde M,e_B\rangle)\cong L^2(\tilde M)\otimes_BL^2(\tilde M)\cong L^2(M)\otimes_B(L^2(\tilde M)\oplus \mathcal H)$, as $M$-$M$ bimodules.  The second part of Lemma \ref{gap} now implies that $S$ is a non-amenability set for $M$ relative to $B$ inside $\tilde M$. In particular,  $M$ is not amenable relative to $B$ inside $\tilde M$. Since $\alpha_t$ leaves $B$ invariant, we derive that $\alpha_t(M)$ is not amenable relative to $B$ inside $\tilde M$, for any $t\in\mathbb R$.

  Assume by contradiction that $\alpha_t(M)\prec_{\tilde M}B\bar{\otimes}L(\mathbb Z)$, for some $t\in\mathbb R$. Since $\alpha_t(M)'\cap\tilde M=\mathbb C1$, by \cite[Remark 2.2]{Io12a} it follows that $\alpha_t(M)$ is amenable relative to $B\bar{\otimes}L(\mathbb Z)$ inside $\tilde M$. Note that $B\bar{\otimes}L(\mathbb Z)$ is amenable relative to $B$ inside $\tilde M$. Indeed, if $u\in L(\mathbb Z)$ is a generating Haar unitary, then the vectors $\xi_n=\frac{1}{\sqrt{n}}\sum_{k=1}^nu^ke_B{u^k}^*\in L^2(\langle \tilde M,e_B\rangle)$ satisfy $\langle x\xi_n,\xi_n\rangle=\tau(x)$, for all $x\in\tilde M$, and $\|y\xi_n-\xi_n y\|_2\rightarrow 0$, for all $y\in B\bar{\otimes}L(\mathbb Z)$.

By combining the last two facts and using \cite[Proposition 2.4 (3)]{OP07} we deduce that $\alpha_t(M)$ is amenable relative to $B$ inside $\tilde M$.  This leads to a contradiction.  \hfill$\square$
\vskip 0.05in
{\bf Claim 4.}
There exists $t_0>0$ such that $\alpha_t(M)\nprec_{\tilde M}M$, for all $t\in (0,t_0)$.

{\it Proof of Claim 4.} Assuming that the claim is false, we can find a sequence $t_n\rightarrow 0$ with $t_n>0$ such that $\alpha_{t_n}(M)\prec_{\tilde M}M$, for all $n\geqslant 1$. On the other hand, Claim 3 gives that $\alpha_{t_n}(M)\nprec_{\tilde M}B$. 
Recall that $\alpha_{t_n}(M)'\cap\tilde M=\mathbb C1$,  $M$ is a factor and $\tilde M=M*_B(B\bar{\otimes}L(\mathbb Z))$.  By combining all these facts,  the proof of \cite[Theorem 5.1]{IPP05}  implies that we can find a unitary operator $v_n\in \tilde M$ such that $v_n\alpha_{t_n}(M)v_n^*\subset M$.

Since $S$ is a non-amenability set for $M$ relative to $B$ inside $\tilde M$,  Lemma \ref{uniform} (2) provides a constant $C>0$ such that for every $n\geqslant 1$ we have $$\|\alpha_{t_n}(x)-E_M(\alpha_{t_n}(x))\|_2\leqslant C\sum_{y\in S}\|\alpha_{t_n}(y)-y\|_2,\;\;\text{for all}\;\; x\in (M)_1.$$

Finally, let $x\in M_0$. Proposition \ref{dimagen} gives that  $\|\frac{\alpha_t(x)-x}{t}-\delta(x)\# s\|_2\rightarrow 0$, as $t\rightarrow 0$. Also, we have that $\|\delta(x)\# s\|_2=\|\delta(x)\|_2$ and $E_M(\delta(x)\#s)=0$. By combining these facts with the last inequality we get that $\|\delta(x)\|_2\leqslant C\sum_{y\in S}\|\delta(y)\|_2$, for all $x\in M_0$, which contradicts Claim 1.\hfill$\square$

{\bf Claim 5.} $M$ does not have property Gamma. 

{\it Proof of Claim 5.} Assume by contradiction that $M$ has property Gamma and let $t\in (0,t_0)$, where $t_0$ is given by Claim 4. Proposition \ref{mixing} (2) then implies that $M'\cap M^{\omega}\nprec_{\tilde M^{\omega}}B^{\omega}$.
Since $B$ is invariant under $\alpha_{t}$, we get that $\alpha_{t}(M)'\cap\tilde M^{\omega}\nprec_{\tilde M^{\omega}}B^{\omega}.$ Also, by the above claims we have that $\alpha_t(M)\nprec_{\tilde M}B\bar{\otimes}L(\mathbb Z)$ and $\alpha_{t}(M)\nprec_{\tilde M}M$.

By  applying \cite[Theorem 6.3]{Io12a} to the inclusion $\alpha_{t}(M)\subset \tilde M=M*_{B}(B\bar{\otimes}L(\mathbb Z))$, we deduce that 
$\alpha_{t}(M)p$ is amenable relative to $B$ inside $\tilde M$, for a non-zero projection $p\in\alpha_{t}(M)'\cap\tilde M$. 
Since $\alpha_t(M)'\cap\tilde M=\mathbb C1$, this would imply that $\alpha_t(M)$ is amenable relative to $B$ inside $\tilde M$, which is false, by Claim 3.\hfill$\square$

We are now ready to prove the conclusion of Theorem \ref{nocartan}. Thus, assume by contradiction that $M$  has a Cartan subalgebra $A$. 
Let $t\in (0,t_0)$. Note that $\alpha_{t}(M)\subset\mathcal N_{\tilde M}(\alpha_{t}(A))''$ and therefore $\alpha_{t}(M)\subset\mathcal N_{\tilde M}(\alpha_{t}(A))''$. By combining Claims 2 and 5, we get that $M'\cap\tilde M^{\omega}=\mathbb C1$.
Thus, we also have that 
  $\alpha_{t}(M)'\cap\tilde M^{\omega}=\mathbb C1$. Altogether, we are in position to apply Theorem \ref{general} (in the case $Q=\mathbb C1$) and deduce that one of the following conditions holds:

\begin{enumerate}
\item $\alpha_{t}(A)\prec_{\tilde M}B$.
\item $\alpha_{t}(M)\prec_{\tilde M}M$. 
\item $\alpha_{t}(M)\prec_{\tilde M}B\bar{\otimes}L(\mathbb Z)$.
\item $\alpha_{t}(M)$ is amenable relative to $B$ inside $\tilde M$.

\end{enumerate}

Since $\alpha_t$ leaves $B$ invariant, condition (1)  implies that $A\prec_{\tilde M}B$. Since $A\subset M$, \cite[Theorem 1.2.1]{IPP05} implies that $A\prec_{M}B$.  This however cannot happen by Proposition \ref{mixing} (1).
 Since conditions (2)-(4) are also false as shown above, we get a contradiction.
\hfill$\square$

\section{Algebraic cocycles and uniqueness of Cartan subalgebras}

In this final section we first prove a slightly more general form of Theorem \ref{unique} and then derive Corollary \ref{HNN}.

\begin{theorem}\label{unique2} Let $\Gamma$ be a group satisfying all the assumptions from Theorem \ref{unique}.  Assume that $\Gamma'$ is a group which admits a finite normal subgroup $N$ such that $\Gamma'/N\cong\Gamma$.

Then $L^{\infty}(X)$ is the unique Cartan subalgebra of $L^{\infty}(X)\rtimes\Gamma'$, up to unitary conjugacy, for any free ergodic probability measure preserving action $\Gamma'\curvearrowright (X,\mu)$.

\end{theorem}

{\it Proof.} Let $\Lambda<\Gamma$ be a subgroup and $b:\Gamma\rightarrow\mathbb C(\Gamma/\Lambda)$ a cocycle satisfying the hypothesis of Theorem \ref{unique}. After replacing $b$ with its real or imaginary part, we may assume that we have $b(\Gamma)\subset\mathbb R(\Gamma/\Lambda)$.
For $g\in\Gamma$, write $b(g)=\sum_{h\Lambda\in\Gamma/\Lambda}c_{g,h\Lambda}\delta_{h\Lambda}$, where $c_{g,h\Lambda}\in\mathbb R$.

Consider the isometry $V:\ell^2(\Gamma/\Lambda)\rightarrow L^2(\langle L(\Gamma),e_{L(\Lambda)}\rangle)$ given by $V(\delta_{h\Lambda})=u_he_{L(\Lambda)}u_h^*$. Then $V(\pi(g)\xi)=u_g\xi u_g^*$, where $\pi:\Gamma\rightarrow\ell^2(\Gamma/\Lambda)$ is the quasi-regular representation $\pi(g)(\delta_{h\Lambda})=\delta_{gh\Lambda}$.

We define $\delta:\mathbb C\Gamma\rightarrow L^2(\langle L(\Gamma),e_{L(\Lambda)}\rangle)$ by putting  $\delta(u_g)=i\;V(b(g))u_g^*$, or, explicitely, $$\delta(u_g)=i\hskip 0.02in \sum_{h\Lambda\in\Gamma/\Lambda}c_{g,h\Lambda}\;u_he_{L(\Lambda)}u_h^*\;u_g,\;\;\text{for all}\;\;g\in\Gamma.$$

Then it is easy to see that $\delta$ is a  real derivation. 
Since $b_{|\Lambda}\equiv 0$, we have that $\delta_{|\mathbb C\Lambda}\equiv 0$.
As for every $g\in\Gamma$ we have that $Tr(\delta(u_g)e_{L(\Lambda)})=\delta_{g,e}c_{e,e\Lambda}=\delta_{g,e}\langle b(e),\delta_{e\Lambda}\rangle=0$, it follows that $\delta^*(e_{L(\Lambda)})=0$. Since $b(\Gamma_n)\subset\mathbb C(\Gamma_n\Lambda/\Lambda)$, we get that $\delta(\mathbb C\Gamma_n)\subset$ span$(\mathbb C\Gamma_ne_{L(\Lambda)}\mathbb C\Gamma_n)$, for all $n\geqslant 1$. Since $\mathbb C\Gamma_n$ is finitely generated and $\mathbb C\Gamma=\cup_{n\geqslant 1}\mathbb C\Gamma_n$, we are in position to apply Proposition \ref{dimagen}.

Let $\tilde\Gamma=\Gamma*_{\Lambda}(\Lambda\times\mathbb Z)$ and $s\in L(\mathbb Z)$ a generating semicircular element. Proposition \ref{dimagen}  provides a one-parameter group of automorphisms $\{\alpha_t\}_{t\in\mathbb R}$ of $L(\tilde\Gamma)$ satisfying $\|\frac{\alpha_t(u_g)-u_g}{t}-\delta(u_g)\# s\|_2\rightarrow 0$, for all $g\in\Gamma$. Since $\delta_{|\mathbb C\Lambda}\equiv 0$ and $\delta^*(e_{L(\Lambda)})=0$, we have that $\alpha_t(x)=x$, for all $x\in L(\Lambda\times\mathbb Z)$.

For $f\in\ell^{\infty}(\Gamma/\Lambda)$ and $h\in\Gamma$, we define $\sigma(h)(f)\in\ell^{\infty}(\Gamma/\Lambda)$ by letting $\sigma(h)(f)(g\Lambda)=f(h^{-1}g\Lambda)$.
Since $\Lambda$ is not co-amenable in $\Gamma$, there exists a finite set $S\subset\Gamma$ such that there is no $\sigma(S)$-invariant state on $\ell^{\infty}(\Gamma/\Lambda)$.
Consider the unital $*$-homomorphism $\rho:\ell^{\infty}(\Gamma/\Lambda)\rightarrow \langle L(\Gamma),e_{L(\Lambda)}\rangle$ given by $\rho(f)=\sum_{g\Lambda\in\Gamma/\Lambda}f(g\Lambda)u_ge_{L(\Lambda)}u_g^*.$  
Then $\rho(\sigma(h)(f))=u_h\rho(f)u_h^*$,  for all $h\in\Gamma$. Thus, there is no $S$-central state on $\langle L(\Gamma),e_{L(\Lambda)}\rangle$ and hence $S$ is a non-amenability set for $L(\Gamma)$ relative to $L(\Lambda)$.

Since $\delta$ is unbounded and $S\subset\Gamma$ is a non-amenability set for $L(\Gamma)$ relative to $L(\Lambda)$,   Claims 1-3 from the proof of Theorem \ref{nocartan} (applied here verbatim  in the case $M=L(\Gamma)$ and $B=L(\Lambda)$) give the following:

\begin{itemize}
\item $L(\Gamma)'\cap L(\tilde\Gamma)^{\omega}\subset L(\Gamma)^{\omega}$.
\item $\alpha_t(L(\Gamma))$ is not amenable relative to $L(\Lambda)$ inside $L(\tilde\Gamma)$, for any $t\in\mathbb R$.
\item $\alpha_t(\Gamma)\nprec_{L(\tilde\Gamma)}L(\Lambda\times\mathbb Z)$, for any $t\in\mathbb R$.
\item there exists $t_0>0$ such that $\alpha_t(L(\Gamma))\nprec_{L(\tilde\Gamma)}L(\Gamma)$, for any $t\in (0,t_0)$.
\end{itemize}

Now, let $\Gamma'\curvearrowright (X,\mu)$ be a free ergodic p.m.p. action. Define $M=L^{\infty}(X)\rtimes\Gamma'$ and let $A$ be a Cartan subalgebra of $M$.  We want to show that $A$ is unitarily conjugate to $L^{\infty}(X)$.

 To this end, let $\Delta:M\rightarrow M\bar{\otimes}L(\Gamma)$ be the $*$-homomorphism given by $\Delta(au_g)=au_g\otimes u_{p(g)}$, for all $a\in L^{\infty}(X)$ and $g\in\Gamma'$ \cite{PV09}. Here, $p:\Gamma'\rightarrow\Gamma$ denotes the quotient homomorphism. 
Further, we  fix $t\in (0,t_0)$ and define $\theta_t=(\text{id}_M\otimes\alpha_t)\circ\Delta:M\rightarrow M\bar{\otimes}L(\tilde\Gamma)$.

Then $\theta_t(M)\subset\mathcal N_{M\bar{\otimes}L(\tilde\Gamma)}(\theta_t(A))''$, therefore $u_g\otimes\alpha_t(u_{p(g)})\in \mathcal N_{M\bar{\otimes}L(\tilde\Gamma)}(\theta_t(A))''$, for all $g\in\Gamma$. Since $L(\Gamma)'\cap L(\tilde\Gamma)^{\omega}\subset L(\Gamma)^{\omega}$ and $L(\Gamma)$ does not have property Gamma, we have
 $\alpha_t(L(\Gamma))'\cap L(\tilde\Gamma)^{\omega}=\mathbb C1$.  By applying Theorem \ref{general} we conclude  that one of the following conditions holds:

\begin{enumerate}
\item $\theta_t(A)\prec_{M\bar{\otimes}L(\tilde\Gamma)}M\bar{\otimes}L(\Lambda)$.
\item $\alpha_t(L(\Gamma))\prec_{L(\tilde\Gamma)}L(\Gamma)$.
\item $\alpha_t(L(\Gamma))\prec_{L(\tilde\Gamma)}L(\Lambda\times\mathbb Z)$.
\item $\alpha_t(L(\Gamma))$ is amenable relative to $L(\Lambda)$ inside $L(\tilde\Gamma$).
\end{enumerate}
Since conditions (2)-(4) cannot hold by the above, condition (1) must be true.
Since $\alpha_t$ leaves $L(\Lambda)$ invariant,  (1) is equivalent to having $\Delta(A)\prec_{M\bar{\otimes}L(\tilde\Gamma)}M\bar{\otimes}L(\Lambda)$. 
Since $\Delta(A)\subset M\bar{\otimes}L(\Gamma)$ and $L(\tilde\Gamma)=L(\Gamma)*_{L(\Lambda)}L(\Lambda\times\mathbb Z)$, \cite[Theorem 1.2.1]{IPP05} gives that $\Delta(A)\prec_{M\bar{\otimes}L(\Gamma)}M\bar{\otimes}L(\Lambda)$.

By using  \cite[Lemma 7.2]{Io12a} we derive that $A\prec_{M}L^{\infty}(X)\rtimes p^{-1}(\Lambda)$. For $i\in\{1,2,...,m\}$, let $g_i'\in\Gamma'$ such that $p(g_i')=g_i$.  Since $M$ is a factor, \cite[Proposition 8]{HPV10}  implies that $A\prec_{M}L^{\infty}(X)\rtimes(\cap_{i=1}^mg_i'p^{-1}(\Lambda){g_i'}^{-1})$. Since $\cap_{i=1}^m g_i\Lambda g_i^{-1}$ is finite, $\cap_{i=1}^m g_i'p^{-1}(\Lambda){g_i'}^{-1}$ is  also finite.  By combining the last two facts we get that  $A\prec_{M}L^{\infty}(X)$. Since $A$ and $L^{\infty}(X)$ are Cartan subalgebras of $M$, \cite[Theorem A.1]{Po01} yields that they are unitarily conjugate (see also \cite[Theorem C.3]{Va06}).
\hfill$\square$

Turning to the proof of Corollary \ref{HNN}, let us first establish the following technical result.

\begin{lemma}\label{inner} Let $G$ be a countable group, $\Lambda<G$ be a subgroup, and $\theta:\Lambda\rightarrow G$ be an injective group homomorphism. Assume that $\Lambda\not=G$ and $\theta(\Lambda)\not=G$. Denote by $\Gamma=HNN(G,\Lambda,\theta)$ the corresponding HNN extension.  Then we have 

\begin{enumerate}
\item If  $\cap_{i=1}^mh_i\Lambda h_i^*=\{e\}$, for some $h_1,h_2,...,h_m\in\Gamma$, then $\Gamma$ is not inner amenable.
\item $\Lambda$ is not co-amenable in $\Gamma$.
\end{enumerate}

\begin{remark}
If $\Lambda=G$ or $\theta(\Lambda)=G$, then $\Lambda$ {\it is} co-amenable in $\Gamma$ by \cite[Proposition 2]{MP03}.
\end{remark}

\end{lemma}
\begin{proof}
Recall that $\Gamma=\langle G,t\;|\; t\lambda t^{-1}=\theta(\lambda),\forall\lambda\in\Lambda\rangle$, where $t$ is the so-called  stable letter.
Let $A\subset G$ and $B\subset G$ be sets of representatives of the left cosets of $\Lambda$ and $\theta(\Lambda)$ in $G$, respectively. We assume that $e\in A\cap B$.
Since $\Lambda\not=G\not=\theta(\Lambda)$, we can find $a\in A\setminus\{e\}$ and $b\in B\setminus\{e\}$.

Below, we will implicitly use the {\it normal form theorem}   \cite[Chapter IV, Theorem 2.1]{LS77}: every  $g\in\Gamma$ can be uniquely written as a product $g=g_nt^{\varepsilon_n}g_{n-1}...g_1t^{\varepsilon_1}g_{0}$, for some $g_0,g_1,...,g_{n}\in G$ and $\varepsilon_1,...,\varepsilon_n\in\{-1,1\}$ such that $g_i\in A$, if $\varepsilon_i=-1$, and  $g_i\in B$, if $\varepsilon_i=1$, for all $i\in\{1,2,...,n\}$, and that there is no consecutive subsequence $t^{\varepsilon},1,t^{-\varepsilon}$ within the sequence $g_n,t^{\varepsilon_n},...,t^{\varepsilon_1},g_0$.

(1) Assume by contradiction that $\Gamma$ is inner amenable and let $\phi:\ell^{\infty}(\Gamma\setminus\{e\})\rightarrow\mathbb C$ be a state which is invariant under the conjugation action of $\Gamma$. For a subset $S\subset\Gamma$, we denote $m(S)=\phi(1_{S\setminus\{e\}})$. 

Denote by $S$  the set of all $g=g_nt^{\varepsilon_n}g_{n-1}...g_1t^{\varepsilon_1}g_{0}\in\Gamma$ (represented in normal form) such that $n\geqslant 1$ and $g_n\not=e$.
Also, denote by $U$ (respectively, $V$) the set of $g\in\Gamma$ such that $n\geqslant 1$, $g_n=e$ and $\varepsilon_n=-1$ (respectively, $\varepsilon_n=1$). Then we have that $$t^{-1}St\subset U,\;\;\; tSt^{-1}\subset V, \;\;\; aUa^{-1}\subset S,\;\;\; bVb^{-1}\subset S,\;\;\;\text{and}\;\;\; aUa^{-1}\cap bVb^{-1}=\emptyset.$$

Since $m$ is a $\Gamma$-invariant finitely additive measure, these inclusions imply that $m(S)\leqslant m(U)$, $m(S)\leqslant m(V)$ and $m(U)+m(V)\leqslant m(S)$. From this we deduce that $m(S)=m(U)=m(V)=0$. Since $S\cup U\cup V\cup G=\Gamma$ and $m(\Gamma)=1$, we conclude that $m(G)=1$.

Since $t^{-1}Gt\cap G=\Lambda$, we further get that $m(\Lambda)=1$. Finally, since $\cap_{i=1}^mh_i\Lambda h_i^*=\{e\}$, we derive that $m(\{e\})=1$. This contradicts the fact that $m(\{e\})=0$.
\vskip 0.05in
(2) Assume by contradiction that $\Lambda$ is co-amenable inside $\Gamma$ and let $\phi:\ell^{\infty}(\Gamma/\Lambda)\rightarrow\mathbb C$ be a $\Gamma$-invariant state. For a subset $S\subset \Gamma/\Lambda$, we denote $m(S)=\phi({1_S})$. Also, we denote by  $\pi:\Gamma\rightarrow\Gamma/\Lambda$ the canonical projection, and still consider $U,V,S$ as in (1). 

Next, we have that $t^{-1}S\subset U,tS\subset V,aU\subset S$ and $bV\subset S$. Moreover, $a\pi(U)\cap b\pi(V)=\emptyset$. Since $\pi$ is $\Gamma$-equivariant and 
$m$ is a finitely additive $\Gamma$-invariant measure, it follows as above that $m(\pi(S))=m(\pi(U))=m(\pi(V))=0$. Since $\pi(S)\cup \pi(U)\cup \pi(V)\cup G/\Lambda=\Gamma/\Lambda$ and $m(\Gamma/\Lambda)=1$, we conclude that $m(G/\Lambda)=1$. Finally, since $tG/\Lambda\cap G/\Lambda=\emptyset$, we would get that $m(\emptyset)=1$, which gives the desired contradiction.
\end{proof}

\subsection{Proof of Corollary \ref{HNN}}  Let $\Gamma=\text{HNN}(G,\Lambda,\theta)$. Since $\cap_{i=1}^mg_i\Lambda g_i^{-1}$ is finite, it follows that $N=\cap_{g\in\Gamma}g\Lambda g^{-1}$ is a finite normal subgroup of $\Gamma$ such that $N<\Lambda$. Moreover, we can find $h_1,h_2,...,h_n\in\Gamma$ such that $N=\cap_{j=1}^nh_j\Lambda h_j^{-1}$. Denote $\Gamma_0=\Gamma/N$, $G_0=G/N$, $\Lambda_0=\Lambda/N$ and let  $p:\Gamma\rightarrow \Gamma_0$ be the quotient homomorphism.

Since the stable letter $t$ normalizes $N$, if $g\in\Lambda$ then $\theta(g)=tgt^{-1}\in N$ if and only if $g\in N$. Therefore, $\theta:\Lambda\rightarrow G$ descends to an injective group homomorphism $\theta_0:\Lambda_0\rightarrow G_0$. Moreover, we have that $\Gamma_0$ is naturally isomorphic to HNN$(G_0,\Lambda_0,\theta_0)$ and $\cap_{j=1}^np(h_j)\Lambda_0p(h_j)^{-1}=\{e\}$.
Since $\Lambda_0\not=G_0$ and $\theta_0(\Lambda_0)\not=G_0$, by Lemma \ref{inner} we get that $\Gamma_0$ is not inner amenable (hence $L(\Gamma_0)$ is a II$_1$ factor without property Gamma) and $\Lambda_0$ is not co-amenable inside $\Gamma_0$.

Next, we define $b:\Gamma_0=\text{HNN}(G_0,\Lambda_0,\theta_0)\rightarrow\mathbb C(\Gamma_0/\Lambda_0)$  by letting $b(g)=0$, for all $g\in G_0$, and $b(t)=t\Lambda_0$, where $t\in\Gamma_0$ is the stable letter. Then $b$ is an unbounded cocycle. To see that $b$ is unbounded, just note that if $g\in G_0\setminus\Lambda_0$, then $\|b((gt)^n)\|_2=\sqrt{n}$, for all $n\geqslant 0$.

Finally, let $\{G_n\}_{n\geqslant 1}$ be a sequence of finitely generated subgroups of $G_0$ such that $G_0=\cup_{n\geqslant 1}G_n$. Let $\Gamma_n<\Gamma_0$ be the subgroup generated by $G_n$ and $t$. Then $\Gamma_n$ is finitely generated, $\Gamma_n\subset\Gamma_{n+1}$ and $b(\Gamma_n)\subset\mathbb C(\Gamma_n\Lambda/\Lambda)$, for all $n$. Moreover,  $\cup_{n\geqslant 1}\Gamma_n=\Gamma_0$. Altogether, we can apply Theorem \ref{unique2} to get the conclusion. \hfill$\square$

\end{document}